\documentclass[11pt]{amsart}

\usepackage[colorlinks = true,
            linkcolor = red,
            urlcolor  = magenta,
            citecolor = black,
            anchorcolor = blue]{hyperref}
\usepackage{color}
\usepackage[shortalphabetic,initials,nobysame]{amsrefs}
\usepackage{upgreek}
\usepackage{tikz}
\usepackage[applemac]{inputenc} 
\usetikzlibrary{arrows}
\usepackage{xcolor}
\usepackage[all]{xy}
\usepackage{graphicx}
\usepackage{wrapfig}
\usepackage{yfonts}
\usepackage{graphicx}
\usepackage{amssymb}
\usepackage{epstopdf}
\usepackage{mathrsfs}
\usepackage[mathcal]{eucal}
\usepackage{bm}

\renewcommand{\eprint}[1]{\href{https://arxiv.org/abs/#1}{#1}}

\BibSpec{article}{%
    +{}  {\PrintAuthors}                {author}
    +{,} { \textit}                     {title}
    +{.} { }                            {part}
    +{:} { \textit}                     {subtitle}
    +{,} { \PrintContributions}         {contribution}
    +{.} { \PrintPartials}              {partial}
    +{,} { }                            {journal}
    +{}  { \textbf}                     {volume}
    +{}  { \PrintDatePV}                {date}
    +{,} { \issuetext}                  {number}
    +{,} { \eprintpages}                {pages}
    +{,} { }                            {status}
    +{,} { \DOI}                   {doi}
    +{,} { \eprint}        {eprint}
    +{}  { \parenthesize}               {language}
    +{}  { \PrintTranslation}           {translation}
    +{;} { \PrintReprint}               {reprint}
    +{.} { }                            {note}
    +{.} {}                             {transition}
    +{}  {\SentenceSpace \PrintReviews} {review}
}

\def\O {{\mathcal{O}}}

\newcommand{\ard}{{\mathbf{d}}}
\newcommand{\CappedVertexTau}{{\begin{tikzpicture}
\draw [ultra thick] (0,0) -- (2,0);
\draw [fill] (2,0) circle [radius=0.1];
\draw [ultra thick] (0.1,0.2) to [out=210,in=150] (0.1,-0.2);
\node at (2.3,0) {$\tau$};
\end{tikzpicture}
}}
\newcommand{\VertexTau}{{\begin{tikzpicture}
\draw [ultra thick] (0.13,0) -- (2,0);
\draw [fill] (2,0) circle [radius=0.1];
\draw [very thick] (0,0) circle [radius=0.13];
\node at (2.3,0) {$\tau$};
\end{tikzpicture}}}
\newcommand{\CappedVertex}{{\begin{tikzpicture}
\draw [ultra thick] (0,0) -- (2,0);
\draw [fill] (0,0) circle [radius=0.1];
\draw [ultra thick] (1.9,0.2) to [out=330,in=30] (1.9,-0.2);
\end{tikzpicture}
}}

\def\fp{ {\textbf{p}}  }

\newcommand{\Complex}{\mathbb{C}}

\makeatletter
\newcommand{\ellSN}{\mathop{\operator@font sn}\nolimits}
\newcommand{\ellCN}{\mathop{\operator@font cn}\nolimits}
\newcommand{\ellDN}{\mathop{\operator@font dn}\nolimits}
\newcommand{\ellAM}{\mathop{\operator@font am}\nolimits}
\newcommand{\ellK}{\mathop{\smash{\operator@font K}\vphantom{a}}\nolimits}
\newcommand{\ellE}{\mathop{\smash{\operator@font E}\vphantom{a}}\nolimits}
\makeatother

\ifx\genfracQflkaj

\else

\fi

\newcommand{\beq}{\begin{equation}}
\newcommand{\eeq}{\end{equation}}

\makeatletter
\def\mr@ignsp#1 {\ifx\:#1\@empty\else #1\expandafter\mr@ignsp\fi}%
\newcommand{\multiref}[1]{\begingroup
\xdef\mr@no@sparg{\expandafter\mr@ignsp#1 \: }%
\def\mr@comma{}%
\@for\mr@refs:=\mr@no@sparg\do{\mr@comma\def\mr@comma{,}\ref{\mr@refs}}%
\endgroup}
\makeatother

\newcommand{\hypref}[2]{\ifx\href\asklfhas #2\else\href{#1}{#2}\fi}
\newcommand{\Secref}[1]{Section~\multiref{#1}}
\newcommand{\secref}[1]{Sec.~\multiref{#1}}

\renewcommand{\eqref}[1]{(\multiref{#1})}

\def\[{\begin{equation}}
\def\]{\end{equation}}
\def\<{\begin{eqnarray}}
\def\>{\end{eqnarray}}

\newtheorem{theorem}{Theorem}[section]
\newtheorem{lemma}[theorem]{Lemma}
\newtheorem{proposition}[theorem]{Proposition}
\newtheorem{corollary}[theorem]{Corollary}

\newtheorem{definition}[theorem]{Definition}

\ifx\href\asklfhas\newcommand{\href}[2]{#2}\fi

\title{Quantum K-theory of Quiver Varieties and Many-Body Systems}

\author{Peter Koroteev}
\author{Petr P. Pushkar}
\author{Andrey V. Smirnov}
\author{Anton M. Zeitlin}
\address{\newline Peter Koroteev
\newline
Department of Mathematics,\newline
University of California at Davis,\newline
Mathematical Sciences Building,\newline
One Shields Ave,\newline
University of California,\newline
Davis, CA 95616}

\address{\newline
Petr P. Pushkar,\newline
Department of Mathematics, \newline
Columbia University,\newline
Room 509, MC 4406,\newline
2990 Broadway,\newline
New York, NY 10027,\newline
pushkar@math.columbia.edu}

\address{\newline
Andrey Smirnov,\newline
Department of Mathematics, \newline
University of California at Berkeley;\newline
Institute for Problems of \newline
Information Transmission \newline
Bolshoy Karetny 19, Moscow 127994, Russia\newline
smirnov@math.berekeley.edu}

\address{
\newline
Anton M. Zeitlin,\newline
Department of Mathematics,\newline
Louisiana State University,\newline
303 Lockett Hall, \newline
Baton Rouge, LA 70803;\newline
IPME RAS, V.O. Bolshoj pr., 61, 199178,\newline
St. Petersburg\newline
zeitlin@lsu.edu,\newline
http://math.lsu.edu/$\sim$zeitlin \newline
http://www.ipme.ru/zam.html  }

\begin{document}
\maketitle

\begin{abstract}
We define quantum equivariant K-theory of Nakajima quiver varieties. We discuss type A in detail as well as its connections with quantum XXZ spin chains and trigonometric Ruijsenaars-Schneider models. Finally we study a limit which produces a K-theoretic version of  results of Givental and Kim, connecting quantum geometry of flag varieties and Toda lattice.
\end{abstract}

\setcounter{tocdepth}{1}
\tableofcontents

\section{Introduction}\label{Sec:Intro}
\subsection{Some Prehistory and Earlier Results} 
The seminal papers of Nekrasov and Shatash\-vili \cite{Nekrasov:2009ui,nekrasov:2009uh} paved the road for close interactions between quantum geometry of certain class of algebraic varieties and quantum integrable systems. Early signs of such a fruitful collaboration between quantum cohomology/quantum K-theory and integrability were noted in mathematics literature in the works of Givental et al \cite{givental1995, 2001math8105G}.

The ideas outlined in these articles gave rise to new developments \cite{Braverman:2010ei,Braverman:2011si, 2012arXiv1211.1287M} followed by other important results, see e.g. \cite{Rimanyi:2014ef,Gorbounov:2014to,Gorbounov:2014fs,Okounkov:2015aa,Okounkov:2016sya,nekrasov:2013aa}.

Recently the basic example, considered in \cite{Nekrasov:2009ui,nekrasov:2009uh} in the physical context of 3d gauge theories, was described from mathematical point of view \cite{Pushkar:2016qvw}. 

In particular, the relation between quantum equivariant K-theory of cotangent bundles to Grassmannians and the so-called XXZ model (see e.g. \cite{Bogolyubov:1993qism,Reshetikhin:2010si}) was fully examined.
The Hilbert space of the XXZ spin chain is identified with the space of equivariant localized K-theory of disjoint union of $T^*Gr(k,n)$ for all $k$ and fixed $n$, considered in the basis of fixed points. 
Using a different method than in standard Gromov-Witten-inspired approach to quantum products, the quantum K-theory ring was defined, as well as the generators using the theory of quasimaps to GIT quotients \cite{Okounkov:2015aa}, \cite{Ciocan-Fontanine:2011tg}. 
Such generators of the quantum K-theory ring, which in \cite{Pushkar:2016qvw} were called quantum tautological bundles  are the deformations (via K\"ahler parameter) of the exterior powers of these tautological bundles. It was shown that their eigenvalues are the symmetric functions of roots of Bethe Ansatz equations. The generating function for such quantum tautological bundles is known in the theory of integrable systems as Baxter $Q$-operator which contains information about the spectrum of genuine physical Hamiltonians.

\subsection{Main Results and the Structure of the Paper} 
The construction of \cite{Pushkar:2016qvw} can certainly be extended beyond Grassmannians to a large class of Nakajima quiver varieties and this is what the first part of the current work is about. 

In Section \ref{Sec:QuantumK} we review and generalize main concepts of \cite{Pushkar:2016qvw} to a general situation. In Subsection \ref{ClassK} we remind basic notions of Nakajima quiver varieties as GIT quotients and their equivariant K-theory. Subsection \ref{SubsecQuasimap}  is devoted to a brief review of theory of nonsingular and relative quasimaps to quiver varieties. Unlike stable maps, the quasimap is a combination of a certain vector bundle on a base curve, together with its section, 
which uses the presentation of Nakajima quiver variety is a GIT quotient. 
That allows us to define in Subsections \ref{PNVSGO} and \ref{Subsec:QuantumKring} two important notions. The first one is the notion of a quantum tautological class, defined using pushforwards via evaluation map with a relative condition, as a deformation of the corresponding equivariant K-theory tautological class. The second one is the deformed product on equivariant K-theory.  In this paper we will refer to the latter as the \textit{quantum product} and the resulting unital ring will be referred to as \textit{quantum K-theory ring}. 
This is different from a standard notion of quantum products defined using stable map theory in the K-theoretic analogue of Gromov-Witten theory. 
In the end of Section \ref{Subsec:QuantumKring} we note, that the quantum tautological classes generate the entire quantum K-theory ring.

We would like to emphasize that in the standard K-theoretic version of Gromov-Witten approach to flag varieties (see e.g. recent results \cite{act1}, \cite{act2}), the analogue of our deformed product, known as a small quantum product, is determined by the deformation of the structure constants. 

Then it is a formidable task in describing the quantum K-ring using generators and relations to verify whether the structure constants are polynomials in K\"ahler parameters -- the property which is given for granted in the quantum cohomology.  Here we are free of these issue and our quantum classes are generators \textit{a priori}.

In Subsection \ref{Subsec:Vertex} the most important tools for the computations in our quantum K-theoretic framework are introduced, known as \textit{vertex} functions. They can be of two types, \textit{bare} and \textit{capped}. These are objects, very close to quantum tautological classes, namely they are equivariant K-theory classes (localized K-theory classes for bare vertices) defined as equivariant pushforwards with nonsingular and relative conditions correspondingly, so that extra equivariant parameter is introduced on a base curve. This equivariant parameter plays a major role in our approach. Namely, the \textit{capping  operator} which relates these two types of vertex functions, satisfies a {\it difference equation}, which is a central topic of \cite{Okounkov:2016sya} as a part of a bigger system of difference equations involving quantum Knizhnik-Zamolodchikov equations \cite{Frenkel:92qkz}.

In Section \ref{Subsec:COandDE} we restrict ourselves to the subclass of Nakajima varieties, such that the set of fixed points under the action of equivariant torus is finite, which includes (partial) flag varieties. Using that in the end of  Section \ref{Sec:QuantumK} we derive, by generalizing the results of  \cite{Pushkar:2016qvw}, the explicit formula for the eigenvalues of multiplication operators on quantum tautological classes via the asymptotics of vertex functions, when equivariant parameter on a base curve is close to identity.

From \Secref{Sec:QuantumKPartialFlags} onwards we restrict ourselves to our main example the cotangent bundles to (partial) flag varieties, which form a subclass of Nakajima varieties for quivers of type $A_n$. In this case one can identify the localized K-theory of all possible cotangent bundles to partial flag varieties for given $n$ with the Hilbert space of XXZ $sl(n)$ model. 
We explicitly compute the vertex functions in Subsection \ref{Subsec:BareVertex}, so that in Subsection \ref{Subsec:BEandBO} we arrive to our first important theorem, which in short can be restated as follows (for explicit formulas, see Theorem \ref{betheth}):\\

\noindent {\bf Theorem 1a.} {\it The eigenvalues of the operators of multiplication by quantum tautological classes are the symmetric functions of Bethe roots, the solutions of Bethe ansatz equations for $sl(n)$ XXZ spin chain.}\\

This statement is a generalization of similar statement for Grassmannians and $sl(2)$ XXZ spin chain in \cite{Pushkar:2016qvw}.

It makes sense to think about the generating functions for quantum tautological bundles, corresponding to exterior powers of every given tautological bundle. The eigenvalues of the 
resulting operators give generating functions for Bethe roots. In the theory of quantum integrable systems those are known as the {\it Baxter operators}.

We also make the following ``compact reduction" of our constructions. In Subsection \ref{Sec:Compact} we discuss the vertex functions and therefore quantum tautological classes in the case when we count only quasimaps to the compact space of partial flags, suppressing contributions of the fiber. It turns out that the corresponding vertex functions are easy to compute, by sending the equivariant parameter, corresponding to the rescaling of symplectic form to infinity. This leads to the following Theorem  (see 
 (\ref{betheth}) for explicit formulae):\\

\noindent {\bf Theorem 1b.} {\it The eigenvalues of the operators of multiplication by quantum tautological classes on $G/P$ are the symmetric functions of Bethe roots, the solutions of Bethe ansatz equations, generalizations of the ones for 5-vertex model.}\\

 In Section \ref{Sec:XXZ/tRS}, we restrict ourselves to the case of cotangent bundle of complete flag variety, and we describe these K-theory rings using generators and relations by employing the duality between XXZ spin chains and trigonometric Ruijsenaars-Schneider (tRS) models \cite{MR929148,MR887995,MR851627}. 

This brings us back to fundamental papers \cite{givental1995, 2001math8105G}, where connection of quantum geometry and integrability is done through multi-particle systems. Givental and Kim  \cite{givental1995} using their approach described the quantum equivariant cohomology ring of complete flag varieties as an algebra of functions on the phase space of Toda lattice, so that the Hamiltonians are taking fixed values, determined by equivariant parameters, namely the space of regular functions on invariant Lagrangian subvariety of Toda lattice. It was suggested in \cite{2001math8105G} and then in \cite{Braverman:2011si} that the K-theoretic version of these results should involve finite difference (relativistic) Toda system. 

Our main result of Section \ref{Sec:XXZ/tRS} is the following Theorem, which allows to describe the quantum K-theory ring using different generators and relations, via trigonometric Ruijsenaars-Schneider (tRS) models:
(for explicit formulae see Theorem  \ref{Th:KTmain}):\\

\noindent {\bf Theorem 2a.} {\it The quantum K-theory ring of the cotangent bundle of a complete flag variety is an algebra of functions on a certain Lagrangian subvariety of the phase space of tRS model.}\\

In a limiting procedure for $G/P$, which we consider in  \Secref{Sec:tRSToda}, we obtain the result suggested by \cite{givental1995, 2001math8105G} (see Theorem \ref{todath} for explicit formulas): \\

\noindent {\bf Theorem 2b.} {\it The quantum K-theory ring  of a complete flag variety is an algebra of functions on a certain Lagrangian subvariety of the phase space of relativistic Toda lattice.}\\

Such limiting procedure between two integrable systems was discussed by many representation theorists, see e.g. \cite{Etingof:ac,Gerasimov:2008ao,Braverman:2010ei}, which by the results of this paper has a pure geometric flavor.

\subsection{Connections to physics literature and beyond}

In physics literature \cite{Gaiotto:2013bwa,Bullimore:2015fr} our main statements were conjectured, based on connections of Nakajima quiver varieties and 3d supersymmetric gauge theories. In particular, the quantum equivariant K-theory of the cotangent bundle to complete flag variety was described in \cite{Bullimore:2015fr} via trigonometric Ruijsenaars-Schneider system. In addition, as expected, with cotangent fiber being removed in a certain limit, that model reduces to finite difference Toda system.  In the current work we prove these physics conjectures thereby bringing together ideas of Givental-Kim-Lee and Nekrasov-Shatashvili.

It was recently shown \cite{Aganagic:2017be,Aganagic:2017la} that capping operators of quantum K-theory of Nakajima quiver varieties, which satisfy quantum Knizhnik-Zamolodchikov equations, can be represented using vertex functions for certain K-theory classes which correspond to the K-theoretic version of {\it stable basis} \cite{Okounkov:2015aa}. In \cite{Bullimore:2015fr} it was proposed that vertex functions, constructed from supersymmetric gauge theories, are the eigenfunctions of quantum trigonometric Ruijsenaars-Schneider Hamiltonians (Macdonald operators). In this paper we work with classical Hamiltonians and find the connection, via the XXZ spin chain, between these operators and quantum K-theory of Nakajima quiver varieties. Thus there exists a correspondence between quantum Knizhnik-Zamolodchikov equations and equations of motion of trigonometric Ruijsenaars-Schneider model, which was further studied in \cite{Koroteev:2020}.\\

\noindent
\textbf{Acknowledgements.}
First of all we would like to thank Andrei Okounkov for invaluable discussions, advises and sharing with us his fantastic viewpoint on modern quantum geometry.
We are also grateful to D. Korb and Z. Zhou for their interest and comments.

The work of A. Smirnov was supported in part by RFBR grants under numbers
15-02-04175 and 15-01-04217 and in part by NSF grant DMS--2054527. 
The work of P. Koroteev, A.M. Zeitlin and A. Smirnov is supported in part by AMS Simons travel grant.
A. M. Zeitlin is partially supported by Simons Collaboration Grant, Award ID: 578501.

\section{Quantum K-theory}\label{Sec:QuantumK}
\subsection{Classical Equivariant K-theory}\label{ClassK}

In this section we give a brief reminder of the classical equivariant K-theory of Nakajima quiver varieties. For a more detailed introduction to quiver varieties, one can consult \cite{Ginzburg:} and for their study in K-theoretic setting one can look in \cite{Nakajima:2001qg} or \cite{Negut:thesis}.

A quiver is a collection of vertices and oriented edges connecting them ($I$ denotes the set of vertices). A framed quiver is a quiver, where the set of vertices is doubled, and each of the vertices in the added set has an edge going from it to the vertex, whose copy it is. It is common to depict the original vertices by circles, and their copies by squares above them. Here is an example of a framed quiver:

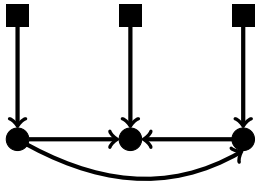
\begin{center}
\begin{tikzpicture}[xscale=1.5, yscale=1.5]
\fill [ultra thick] (0-0.1,1) rectangle (0.1,1.2);
\fill [ultra thick] (1-0.1,1) rectangle (1.1,1.2);
\fill [ultra thick] (2-0.1,1) rectangle (2.1,1.2);
\draw [fill] (0,0) circle [radius=0.1];
\draw [fill] (1,0) circle [radius=0.1];
\draw [fill] (2,0) circle [radius=0.1];
\draw [->, ultra thick] (0.1,0) -- (0.9,0);
\draw [->, ultra thick] (2,0) -- (1.1,0);
\draw [->, ultra thick] (0,1) -- (0,0.1);
\draw [->, ultra thick] (0,0) to [out=330,in=210] (2,-0.1);
\draw [->, ultra thick] (1,1) -- (1,0.1);
\draw [->, ultra thick] (2,1) -- (2,0.1);
\end{tikzpicture}
\end{center}

A representation of a framed quiver is a set of vector spaces $V_i,W_i$, where $V_i$ correspond to original vertices, and $W_i$ correspond to their copies, together with a set of morphisms between these vertices, corresponding to edges of the quiver.

For a given framed quiver, let $R=\text{Rep}(\mathbf{v},\mathbf{w})$ denote the linear space of quiver representation with dimension vectors $\mathbf{v}$ and $\mathbf{w}$, where $\mathbf{v}_i=\text{dim}\ V_i$, $\mathbf{w}_i=\text{dim}\ W_i$. Then the group $G=\prod_i GL(V_i)$ acts on this space in an obvious way. As any cotangent bundle, $T^*R$ has a symplectic structure. This action of $G$ on this space is Hamiltonian with moment map $\mu :T^*R\to \text{Lie}{(G)}^{*}$. Let $L(\mathbf{v},\mathbf{w})=\mu^{-1}(0)$ be the zero locus of the moment map.

The Nakajima variety $X$ corresponding to the quiver is an algebraic symplectic reduction
$$
X=L(\mathbf{v},\mathbf{w})/\!\! /_{\theta}G=L(\mathbf{v},\mathbf{w})_{ss}/G,
$$
depending on a choice of stability parameter $\theta\in {\mathbb{Z}}^I$ (see \cite{Ginzburg:} for a detailed definition). The group
$$\prod GL(Q_{ij})\times
\prod GL({W}_i)\times \mathbb{C}^{\times}_\hbar$$
acts as automorphisms of $X$, coming form its action on $\text{Rep}(\mathbf{v},\mathbf{w})$. Here $Q_{ij}$ stands for the vector space of dimension coming from the incidence matrix of the quiver, i.e. the number of edges between vertices $i$ and $j$,  $\mathbb{C}_{\hbar}$ scales cotangent directions with weight $\hbar$ and therefore symplectic form with weight $\hbar^{-1}$. Let us denote by $\mathsf{T}$ a maximal torus of this group.

The main object of study in this paper will be a certain deformation of the classical equivariant K-theory ring $K_{\mathsf{T}}(X)$. For a Nakajima quiver variety $X$ one can define a set of tautological bundles on it $V_i, W_i, i\in I$ as bundles constructed by applying the associated bundle construction to the G- representations V and W. It follows from this construction, that  all bundles $W_i$ are topologically trivial. Tensorial polynomials of these bundles and their duals generate $K_{\mathsf{T}}(X)$ according to Kirwan's surjectivity theorem, which is recently proven in \cite{McGerty:2016kir}. Let $(\cdot,\cdot)$ be a bilinear form on $K_{\mathsf{T}}(X)$ defined by the following formula
\beq\label{bilf}
(\mathcal{F},\mathcal{G})=\chi(\mathcal{F}\otimes\mathcal{G}\otimes K^{-1/2}),
\eeq
where $K$ is the canonical class and $\chi$ is the equivariant Euler characteristic. Nakajima quiver varieties are a special class of varieties, for which there always exists a square root of the canonical bundle, and it can be chosen canonically from the construction (see Section 6.1 in \cite{Okounkov:2015aa}). The variety $X$ is almost never compact, apart from the cases when it is a point. The locus of fixed points of $\mathsf{T}$, on the other hand, is compact. This allows us to talk about the equivariant Euler characteristic via localization. The necessary extra shift of the bilinear form described above will be explained below.

\subsection{Quasimaps}\label{SubsecQuasimap}
In this section we give a definition of quasimaps and discuss the properties and types of quasimaps we will use.

\begin{definition} \label{quasimap}
	A stable quasimap to a Nakajima quiver variety from a genus 0 curve $\mathcal{D}$ to $X$ relative to points $p_1,\cdots,p_m \in \mathcal{D}$  is given by the following data
	$$
	(\mathcal{C},p_1',\dots,p_m',P,f, \pi),
	$$	
	where
	\begin{itemize}
		\item $\mathcal{C}$ is a connected, at most nodal genus zero projective curve and $p_i'$ are nonsingular points of $\mathcal{C}$,
		\item $P$ is a principal $G$ - bundle over $\mathcal{C}$,
		
		\item f is a section of the fiber bundle 
		\beq \label{qm}
		{\rm p} : P\times_{{G}} (R \oplus R^{*})  \rightarrow \mathcal{C}
		\eeq
		over $\mathcal{C}$ satisfying $\mu =0$, where $R=\text{Rep}(\mathbf{v},\mathbf{w})$ - is a representation of $G$
		defined in Section \ref{ClassK} (the moment map condition is satisfied pointwise, so for every point we can consider the moment map and the image of the section $f$ restricted to every point should be 0),   
		\item $\pi: \mathcal{C} \rightarrow \mathcal{D}$ is a regular map,
		\end{itemize}
		satisfying the following conditions:
		\begin{enumerate}
		\item There is a distinguished component $\mathcal{C}_0$ of $\mathcal{C}$ such that 
		$\pi$ 	restricts to an isomorphism $\pi:\mathcal{C}_0 \cong \mathcal{D} $ and $\pi(\mathcal{C}\setminus \mathcal{C}_{0})$ is zero-dimensional (possibly empty). 
		
		\item $\pi(p_i')=p_i$.
		
		\item $f(p)$ is stable for all $p\in \mathcal{C}\setminus B$ where $B$ is a finite (possibly empty) subset of $\mathcal{C}$. 
		
		\item The set $B$ is disjoint from the nodes and points $p'_1,\dots, p'_m$. 
		
		\item $\omega_{\tilde{\mathcal{C}}}( \sum_i p'_i + \sum_j q_i) \otimes {\mathcal{L}}_{\theta}^{\epsilon}$ is ample for every rational $\epsilon >0$, 
		where 
		${\mathcal{L}}_{\theta}=P\times_{G}\mathbb{C}_{\theta}$ 
		($\theta=\det$ 
		is the character of $G$),
		$\tilde{\mathcal{C}}$ is the closure of $\mathcal{C}\setminus \mathcal{C}_0$ and $q_i$ are the nodes  $\mathcal{C}_0\cap \tilde{\mathcal{C}}$. 
		\end{enumerate}
	
\end{definition}
We call $\mathcal{D}$ the base curve of the quasimap (although for some quasimaps the actual domain might be bigger). Note that it can have one or multiple components. 

Let $(\mathcal{C},p_1',\dots,p_m',P,f, \pi)$ be a quasimap and let
$V_1,V_2,\dots$ be representations of 
$G$ as in Section \ref{ClassK}. Let us denote by 
\beq \label{qmvdef}
\mathscr{V}_i=P\times_{G} V_i \rightarrow \mathcal{C}
\eeq the associated rank $\mathbf{v}_i$ vector bundle over $\mathcal{C}$ and bundles $\mathscr{W}_i$ and $\mathscr{R}$ defined in an analogous way. 
 
\begin{definition} \label{defdegree}
	The degree of a quasimap $(\mathcal{C},p_1',\dots,p_m',P,f, \pi)$ is the vector of degrees of vector bundles 	$\mathscr{V}_i$ associated to it. 
\end{definition}

\begin{definition}
	Let $\textsf{QM}^{\ard}_{{\rm{relative}}, p_1,\cdots,p_m}$ denote the stack parameterizing stable genus zero quasimaps relative to $p_1,\dots,p_m$, (i.e. the data of Definition \ref{quasimap}) of fixed degree ${\ard}$. Two quasimaps are considered isomorphic if there is an isomorphism between the bundles which intertwines the sections.  	
\end{definition}
%
%
For any point on the curve $p\in \mathcal{C}$ we have an evaluation map to the quotient stack $\text{ev}_p : \textsf{QM}^{\ard} \to L(\mathbf{v},\mathbf{w})/G$ defined by
$\text{ev}_p(f)=f(p)$. Note that the quotient stack contains $X$ as an open subset corresponding to locus of semistable points: $$X={\mu}_{ss}^{-1}(0)/G\subset L(\mathbf{v},\mathbf{w})/G.$$
A quasimap $f$ is called nonsingular at $p$ if $f(p)\subset X$ and the quasimap is not relative to $p$.  In short, we conclude that the open subset ${{\textsf{QM}}^\ard}_{\text{nonsing  p}}\subset {{\textsf{QM}}^\ard}$ of quasimaps nonsingular at the given point $p$
is endowed with a natural evaluation map:
\beq
{{\textsf{QM}}^\ard}_{\text{nonsing }\, p} \stackrel{{\text{ev}}_p}{\longrightarrow} X
\eeq
that sends a quasimap to its value at $p$. The moduli space of relative quasimaps ${{\textsf{QM}}^\ard}_{\text{relative} \, p}$ is a resolution of ${\text{ev}}_p$ (or compactification), meaning we have a commutative diagram:
\begin{center}
\begin{tikzpicture}[node distance =5.1em]

  \node (rel) at (2.5,1.5) {${{\textsf{QM}}^\ard}_{\text{relative}\, p}$};
  \node (nonsing) at (0,0) {${{\textsf{QM}}^\ard}_{\text{nonsing}\,  p}$};
  \node (X) at (5,0) {$X$};
  \draw [->] (nonsing) edge node[above]{$\text{ev}_p$} (X);
  \draw [->] (rel) edge node[above]{$\widetilde{\text{ev}}_p$} (X);
  \draw [right hook->] (nonsing) edge  (rel);
    \end{tikzpicture}
\end{center}
with a \textbf{proper} evaluation map $\widetilde{\text{ev}}_p$ from ${{\textsf{QM}}^\ard}_{\text{relative}\, p}$ to $X$. Definition \ref{quasimap}  constructs all the spaces mentioned above, as well as possible combinations with multiple relative points. 

These moduli spaces have a natural action of maximal torus $\mathsf{T}$, lifting its action from $X$. When there are at most two special (relative or nonsingular or marked) points and the base curve is ${\mathbb{P}}^1$ we extend $\mathsf{T}$ by additional torus $\mathbb{C}^{\times}_q$, which scales ${\mathbb{P}}^1$ such that the tangent space $T_{0} {\mathbb{P}}^1$ has character denoted by $q$. We call the full torus by $G=\mathsf{T}\times \mathbb{C}^{\times}_q$.

\subsection{Picture Notations, Virtual Structure and Gluing Operator }\label{PNVSGO}

In this section we introduce some notations and discuss some structures and and properties of quasimap spaces. There are no new results presented in this section, it is more a collection of things we will use to construct the further studied objects. Most definitions and properties presented here are presented in full generality in \cite{Ciocan-Fontanine:2011tg} or in \cite{Okounkov:2015aa}.

\subsubsection{Picture Notation}In the previous section, several different types of quasimap invariants and conditions were introduced. For the quasimaps considered, the base curve is fixed and it is important, which conditions we impose at different points. All this information is hard to read off a formula. This makes it is convenient to use picture notation, introduced by Okounkov in \cite{Okounkov:2015aa}. The picture notation will almost always be accompanied by a formula presentation, as it is not always obvious what exactly is considered (generating function for enumerative invariants/quantum operator or a single invariant). Yet, one can argue that this notation makes it clearer what kind of invariants are considered. Here are picture notations, which will be used in this manuscript:
\vspace{0.2in}

\hspace{45pt}\begin{tikzpicture}
\draw [ultra thick] (0,0) -- (2,0);
\end{tikzpicture}  \hspace{10pt}denotes the base curve ${\mathbb{P}}^1$,

\vspace{0.2in}
\hspace{45pt}\begin{tikzpicture}
\draw [ultra thick] (0,0) -- (1,0);
\draw [fill] (1,0) circle [radius=0.1];
\end{tikzpicture}\hspace{10pt} denotes a marked point (any chosen point on the curve),

\vspace{0.2in}
\hspace{45pt}\begin{tikzpicture}
\draw [ultra thick] (0,0) -- (1,0);
\draw [ultra thick] (0.9,0.2) to [out=330,in=30] (0.9,-0.2);
\end{tikzpicture}\hspace{10pt} denotes a relative point,

\vspace{0.2in}
\hspace{45pt}\begin{tikzpicture}
\draw [ultra thick] (0,0) -- (0.87,0);
\draw [very thick] (1,0) circle [radius=0.13];
\end{tikzpicture}\hspace{10pt} denotes a nonsingular point.

\vspace{0.2in}
\hspace{45pt}\begin{tikzpicture}
\draw [ultra thick] (2,0) to [out=0,in=240] (2.7,0.4);
\draw [ultra thick] (3,0) to [out=180,in=300] (2.3,0.4);
\end{tikzpicture}\hspace{10pt} denotes a node on the base curve.
\vspace{0.2in}

Here is an example of this notation in use:

\vspace{0.1in}
\begin{center}
\CappedVertex
\end{center}

The picture above stands for the following generating function of invariants:

$$\sum\limits_{\ard=\overrightarrow{0}}^{\infty} z^\ard  {\rm{ev}}_{p_2, *}\Big(\textsf{QM}^\ard_{{\rm{relative}} \, p_2}, \widehat{{\O}}_{{\rm{vir}}} \Big).$$

\subsubsection{Virtual Structure} The moduli spaces of quasimaps constructed in the previous section have perfect deformation-obstruction theory \cite{Ciocan-Fontanine:2011tg}. This allows one to construct a tangent virtual bundle $T^{\textrm{vir}}$, a virtual structure sheaf ${\mathcal{O}}_{\rm{vir}}$ and a virtual canonical bundle. For Nakajima quiver varieties the virtual canonical bundle has a natural choice of a square root. Adjusting the virtual structure sheaf by this square root makes it into the symmetrized virtual structure sheaf $\hat{\mathcal{O}}_{\rm{vir}}$. It is this sheaf that we choose for our enumerative invariants. The motivation of such a choice is given in section 3.2 of \cite{Okounkov:2015aa}. In this section we do not intent to give the full construction of the virtual structure sheaf, but we try to describe some of it properties and provide a way for computing it.\\

First of all, we state a formula for the reduced virtual tangent bundle. Let $(\{\mathscr{V}_i\},\{W_i\})$ be the data defining a quasimap. Then the virtual tangent bundle is an equivariant K-theory class, which when restricted to a fixed point in the space of quasimaps is:
\beq
T^{\textrm{vir}}_{(\{\mathscr{V}_i\},\{\mathscr{W}_i\})}
{\textsf{QM}}^{\ard}= 
H^{\bullet}(\mathscr{R}\oplus\hbar\, \mathscr{R}^{*} ) -(1+\hbar) \bigoplus_{i} Ext^{\bullet}(\mathscr{V}_i,\mathscr{V}_i),
\eeq
where the bundle $\mathscr{R}$ is defined as in \ref{quasimap}. Let us address the different terms in this formula:
\begin{itemize}
\item The term $H^{\bullet}(\mathscr{R}\oplus\hbar\, \mathscr{R}^{*} )$ keeps track of deformations and obstructions of the section $f$.
\item The term $-(1+\hbar) \bigoplus_{i} Ext^{\bullet}(\mathscr{V}_i,\mathscr{V}_i)$ accounts for the moment map equations, and for automorphisms and deformations of $\mathscr{V}_i$.
\end{itemize}

As stated above, this virtual tangent bundle comes from a perfect deformation-obstruction theory. This allows one to construct a virtual structure sheaf ${\mathcal{O}}_{\rm{vir}}$ \cite{BF}. The virtual structure sheaf is a K-theoretic analog of the virtual fundamental class in cohomology. It was first proposed by Kontsevich in \cite{10.1007/978-1-4612-4264-2_12}
 and then identified in \cite{BF}. The virtual structure sheaf was used in Y.P. Lee's original approach to quantum K-theory via moduli spaces of stable maps \cite{YPLee}. Later its construction was extended to greater generality in \cite{CFK}.

Having said this, we want to stress that we will only be doing computations with the virtual structure sheaf by using virtual localization formulas, meaning that the provided formula for the virtual tangent bundle is enough for all the computations of this paper.


The symmetrized virtual structure sheaf is defined by:
\beq
\label{virdef}
\hat{\mathcal{O}}_{\rm{vir}}=\mathcal{O}_{\rm{vir}}\otimes {\mathscr{K}^{1/2}_{\textrm{vir}}} q^{\deg(\mathscr{P})/2},
\eeq
where $\mathscr{K}_{\textrm{vir}}={\det}^{-1}T^{\textrm{vir}} \textsf{QM}^{\ard}$ is the virtual canonical bundle and $\mathscr{P}=\mathscr{R}-\bigoplus_{i} Ext^{\bullet}(\mathscr{V}_i,\mathscr{V}_i)$
is the polarization bundle. We do not go into details behind the construction of the square root of the canonical bundle, but yet again address the reader to section 3.2 in \cite{Okounkov:2015aa} for the motivation and section 6.1 for its construction for the space of quasimaps. 

Since we will be using the symmetrized virtual structure sheaf we will need to adjust the standard bilinear form on $K$-theory. That is the reason to for the shift of the bilinear form in (\ref{bilf}).

Finally, all the constructions mentioned above can be generalized to quasimaps nonsingular at a point (by simply restricting sheaves to an open subset), quasimaps relative at a point (see section 6.4 in \cite{Okounkov:2015aa}), as well as any combination of the above conditions to different points. We do not give any formulas for computing virtual structure sheaves for relative conditions, as we will not be explicitly computing any such invariants.

\subsubsection{Gluing Operator} In order to construct the quantum product we need an important element in the
theory of relative quasimaps, namely the gluing operator. As for all operators or enumerative invariants in this paper we will use the following notation for Kahler variables: for a vector $\ard=(d_i)$, $$z^{\ard}=\prod_{i\in I}z_i^{d_i}.$$This is the operator\footnote{In fact, the gluing operator
is a rational function of the quantum parameters $\mathbf{G}\in End (K_{\mathsf{T}}(X))(z)$ and $\mathbf{G}^{-1}$ is also an endomorphism of non-localized $K$ theory (See Section 6.5 in \cite{Okounkov:2015aa})}
$\mathbf{G}\in End (K_{\mathsf{T}}(X))[[z]]$ defined by:
\begin{eqnarray}
&&\mathbf{G}=\sum\limits_{\ard=\overrightarrow{0}}^{\infty} z^\ard ev_{p_1, p_2 *}(\textsf{QM}^{\ard}_{{\rm relative} p_1, p_2  }\hat{\mathcal{O}}_{\rm{vir}})\in K_{\mathsf{T}}^{\otimes 2}(X)[[z]],\\
&&{\rm so ~\ that ~\ the ~\ corresponding ~\ picture ~\ is:} \quad
\begin{tikzpicture}
\draw [ultra thick] (0,0) -- (2,0);
\draw [ultra thick] (0.1,0.2) to [out=210,in=150] (0.1,-0.2);
\draw [ultra thick] (1.9,0.2) to [out=330,in=30] (1.9,-0.2);
\end{tikzpicture}.\nonumber
\end{eqnarray}
It plays an important role in the degeneration formula, see e.g. \cite{Okounkov:2015aa}. Namely, let a smooth curve $\mathcal{C}_{\varepsilon}$ degenerate to a nodal curve:
$$
\mathcal{C}_{0}=\mathcal{C}_{0,1}\cup_p\mathcal{C}_{0,2}.
$$
Here $\mathcal{C}_{0,1}$ and $\mathcal{C}_{0,2}$ are two different components that are glued to each other at point $p$.
The degeneration formula counts quasimaps from $\mathcal{C}_{\varepsilon}$ in terms of relative quasimaps from $\mathcal{C}_{0,1}$ and $\mathcal{C}_{0,2}$, where the relative conditions are imposed at the gluing point $p$. The family of spaces $\textsf{QM}(\mathcal{C}_{\varepsilon}\to X)$ is flat, which means that we can replace curve counts for any $\mathcal{C}_{\varepsilon}$ by $\mathcal{C}_{0}$. In particular, we can replace counts of quasimaps from ${\mathbb{P}}^1$ by a degeneration of it, for example by two copies of ${\mathbb{P}}^1$ glued at a point.

The gluing operator ${\bold{G}} \in \textrm{End} K_{\mathsf{T}}(X)[[z]]$ is the tool that allows us to replace quasimap counts on $\mathcal{C}_{\varepsilon}$ by counts on $\mathcal{C}_{0,1}$ and $\mathcal{C}_{0,2}$, so that the following degeneration formula holds:
$$
\chi(\textsf{QM}(\mathcal{C}_{0}\to X), \hat{\mathcal{O}}_{\rm{vir}}z^{\ard})=\left({\bold{G}}^{-1}\textrm{ev}_{1,*}(\hat{\mathcal{O}}_{\rm{vir}} z^{\ard}),\textrm{ev}_{2,*}(\hat{\mathcal{O}}_{\rm{vir}} z^{\ard})\right).
$$
The corresponding picture interpretation is as follows:
\begin{center}
\begin{tikzpicture}
\draw [ultra thick] (0,0) -- (1.4,0);
\node at (1.7,0) {$=$};
\draw [ultra thick] (2,0) to [out=0,in=240] (2.7,0.4);
\draw [ultra thick] (3,0) to [out=180,in=300] (2.3,0.4);
\node at (3.3,0) {$=$};
\draw [ultra thick] (3.6,0) -- (4.6,0);
\draw [ultra thick] (4.5,0.2) to [out=330,in=30] (4.5,-0.2);
\node at (5.05,0) {${\bold{G}}^{-1}$};
\draw [ultra thick] (5.5,0) -- (6.5,0);
\draw [ultra thick] (0.1+5.5,0.2) to [out=210,in=150] (0.1+5.5,-0.2);
\end{tikzpicture}
\end{center}

\subsection{Quantum $K$-theory Ring}\label{Subsec:QuantumKring}
In this section we define multiplication and important objects of the quantum $K$-theory ring of $X$.

As a vector space quantum $K$-theory ring $QK_{\mathsf{T}}(X)$ is isomorphic to $K_{\mathsf{T}}(X)\otimes\mathbb{C}[[z_{\{i\}}]], \,i\in I$.
\begin{definition}
The element of the quantum $K$-theory
\beq
\hat{\tau}(z)=\sum\limits_{\ard=\overrightarrow{0}}^{\infty} z^\ard  {\rm{ev}}_{p_2, *}\Big(\textsf{QM}^\ard_{{\rm{relative}} \, p_2}, \widehat{{\O}}_{{\rm{vir}}}\tau (\left. \mathscr{V}_i\right|_{p_1}) \Big) \in QK_{\mathsf{T}}(X)
\eeq
is called quantum tautological class corresponding to tensorial polynomial $\tau$ in tautological bundles $\mathcal{V}_i$. In picture notation it will be represented by
\vspace{0.1in}
\begin{center}
\CappedVertexTau
\end{center}
\end{definition}
These classes evaluated at $0$ are equal to the classical tautological classes on $X$ ($\hat{\tau}(0)=\tau$). Note that the definition depends on the tensorial polynomial $\tau$ rather than a class in K theory of $X$.

For any element $\mathcal{F}\in K_{\mathsf{T}}(X)$ the following element
\beq
\sum\limits_{\ard=\overrightarrow{0}}^{\infty}\,z^{\ard} {\text{ev}}_{p_1,p_3\ast}\left({\textsf{QM}}^{\ard}_{p_1,p_2,p_3},\text{ev}^{\ast}_{p_2}({\mathbf{G}}^{-1}\mathcal{F})\widehat{{\O}}_{{\rm{vir}}} \right)\in {K_{\mathsf{T}}(X)}^{\otimes2}[[z]]
\eeq
can be made into an operator from the second copy of ${K_{\mathsf{T}}(X)}$ to the first copy by the bilinear form $(\cdot,\cdot)$ defined above. We define the operator of quantum multiplication by $\mathcal{F}$ to be this operator shifted by ${\mathbf{G}}^{-1}$, i.e
\beq
\label{quantumproduct}
\mathcal{F}\circledast=\sum\limits_{\ard=\overrightarrow{0}}^{\infty}\,z^{\ard} {\text{ev}}_{p_1,p_3\ast}\left({\textsf{QM}}^{\ard}_{p_1,p_2,p_3},\text{ev}^{\ast}_{p_2}({\mathbf{G}}^{-1}\mathcal{F})\widehat{{\O}}_{{\rm{vir}}} \right){\mathbf{G}}^{-1}
\eeq

\begin{definition}
We call $QK_{\mathsf{T}}(X)=K_{\mathsf{T}}(X)[[z]]$ endowed with multiplication (\ref{quantumproduct}), the quantum K-theory ring of $X$. 
\end{definition}

This product enjoys properties similar to the product in quantum cohomology. The proof of the following statement repeats to the proof of the analogous fact for the cotangent bundle to Grassmannian \cite{Pushkar:2016qvw}.

\begin{theorem}
The quantum $K$-theory ring $QK_{\mathsf{T}}(X)$ is a commutative, associative and unital algebra.
\end{theorem}

\noindent
\textbf{Important Assumptions:} From now on we assume that the fixed points set $X^{\mathsf{T}}$ is finite. The classes of fixed points 
are eigenvectors of classical multiplication in $K_{\mathsf{T}}(X)$.
We assume, in addition, that for any two fixed points there exists a line bundle $\mathcal{L}$ for which the corresponding two eigenvalues are distinct. This is indeed the case for our main example in this paper, namely cotangent bundles for partial flag varieties.

After quantum deformation, the eigenvalues of quantum multiplication by $\mathcal{L}$ become power series in the K\"ahler parameters $z$, with the first term given by the classical eigenvalue,
see Lemma \ref{lemmpowe} below. Thus, the eigenvalues remain distinct in a small neighborhood of zero $|z|\ll 1$. Therefore, our assumptions guarantee that the quantum K-theory ring remains diagonalizable in a perhaps deformed basis.  

\vspace{2mm}

\noindent {\bf Remark.}
In general, the situation of degenerate eigenvalues is unavoidable, with Hilbert scheme of $k$ points the complex plane $X=\mathrm{Hilb}^{k}(\mathbb{C}^2)$ as an example. Its Picard group is generated by a single element $\mathscr{O}(1)$ and the corresponding eigenvalues appear with multiplicities. 

\vspace{2mm}

The operators of quantum multiplication by the \textbf{quantum} tautological bundles obey the most natural properties.   
First, given Kirwan's K-theoretic surjectivity theorem, we have the following result.

\begin{proposition}\label{generators}
Quantum tautological classes generate the quantum equivariant  $K$-theory over the quantum equivariant  $K$-theory of a point
$QK_{\mathsf{T}}(\cdot)=\mathbb{C}[a_{m}^{\pm 1}][[z_i]]$ where $a_m$ for $m=1\cdots \dim \mathsf{T}$ are the equivariant parameters of $\mathsf{T}$.
\end{proposition}
\begin{proof}
Since, by Kirwan's $K$-theoretic surjectivity theorem, classical K-theory is generated by tautological classes, the quantum $K$-theory will be generated by quantum tautological classes according to Nakayama's Lemma.
\end{proof}

Second, in contrast with quantum cohomology, the multiplicative identity of the quantum $K$-theory ring does not always coincide with the multiplicative identity of classical $K$-theory (i.e. the structure sheaf ${{\mathcal{O}}}_{X}$):
\begin{proposition}
The multiplicative identity of $QK_{\mathsf{T}}(X)$ is given by $\hat{\bold{1}}(z)$ (i.e. the quantum tautological class for insertion $\tau=1$).   
\end{proposition}
\begin{proof}
The diagrammatic proof given in \cite{Pushkar:2016qvw} can be applied to any Nakajima quiver variety.
\end{proof}

\subsection{Vertex functions}\label{Subsec:Vertex}
The spaces $\textsf{QM}^{\ard}_{\rm{nonsing}\, p_2}$ and $\textsf{QM}^{\ard}_{\rm{relative}\, p_2}$ admit an action of an extra torus $\mathbb{C}^{\times}_q$ which scales the original $\mathbb{P}^1$ keeping points $p_1$ and $p_2$ fixed. Set $\mathsf{T}_q=\mathsf{T}\times \mathbb{C}^{\times}_q$ be the torus acting on these spaces.
\begin{definition}
The element
$$
V^{(\tau)}(z)=\sum\limits_{\ard=\vec{0}}^{\infty} z^{\ard} {\rm{ev}}_{p_2, *}\Big(\textsf{QM}^{\ard}_{{\rm{nonsing}} \, p_2},\widehat{{\O}}_{{\rm{vir}}} \tau (\left.\mathscr{V}_i\right|_{p_1}) \Big) \in  K_{\mathsf{T}_q}(X)_{loc}[[z]]
$$
is called bare vertex with descendent $\tau$.
In picture notation it will be denoted by
\vspace{0.1in}
\begin{center}
\VertexTau
\end{center}
\end{definition}
The space $\textsf{QM}^{\ard}_{\rm{nonsing}\, p_2}$ is not proper (the condition of non-singularity at a point is an open condition), but the set of $\mathsf{T}_q$-fixed points is, hence the bare vertex is singular at $q=1$.

\begin{definition}
The element
$$
\hat{V}^{(\tau)}(z)=\sum\limits_{\ard=\vec{0}}^{\infty} z^{\ard} {\rm{ev}}_{p_2, *}\Big(\textsf{QM}^{\ard}_{{\rm{relative}} \, p_2},\widehat{{\O}}_{{\rm{vir}}} \tau (\left.\mathscr{V}_i\right|_{p_1}) \Big) \in  K_{\mathsf{T}_q}(X)[[z]]
$$
is called capped vertex with descendent $\tau$. In picture notation it will be represented by:
\vspace{0.1in}
\begin{center}
\CappedVertexTau
\end{center}
\end{definition}
Note here, that the definition of the capped vertex and the definition of quantum tautological classes are very similar with the main difference being the spaces they live in. By definition, the quantum tautological classes can be obtained by taking a limit of the capped vertex: $\lim_{q\to 1}\hat{V}^{(\tau)}(z)=\hat{\tau}(z)$. The last limit exists as the coefficients of $\hat{V}^{(\tau)}(z)$ 
are Laurent polynomials in $q$, due to the properness of the evaluation map in the relative case. 

In fact, the following strong statement is known about capped vertex functions.
\begin{theorem}
Power series $\hat{V}^{(\tau)}(z)$ is a Taylor expansion of a rational function in quantum parameters $z$.
\end{theorem}
\begin{proof}
There are two different proofs of this theorem: the first is based on large framing vanishing 
\cite{2016arXiv161201048S}, the second originates from integral representations 
of solutions of quantum difference equations \cite{Aganagic:2017be}. 
\end{proof}
As a corollary, quantum tautological classes $\hat{\tau}(z)$ are rational functions of $z$.

\subsection{Capping Operator and Difference Equation}\label{Subsec:COandDE}
The operator which relates capped and bare vertices, is known as capping operator and is defined as the following class in the localized K-theory:
\begin{equation}
\Psi(z)=\sum\limits_{\ard=0}^{\infty}\, z^{\ard} {\rm ev}_{p_1,p_2,*}
\Big( {\it \textsf{QM}}^{\ard}_{{{\rm{relative}\, p_1}} \atop {{\rm{nonsing}\, p_2}}},
\widehat{{\O}}_{{\rm{vir}}} \Big)
\in K^{\otimes 2}_{\mathsf{T}_q}(X)_{loc}[[z]].
\end{equation}
Bilinear form makes it an operator acting from the second to the first copy of $K_{\mathsf{T}_q}(X)_{loc}[[z]]$.
This operator satisfies the quantum difference equations.  We summarize that in the Theorem below \cite{Okounkov:2015aa}.
\begin{theorem}
1)The capped vertex with descendent $\tau$ is a result of applying of the capping operator to the bare vertex
\begin{equation}
\label{verrel}
\hat{V}^{(\tau)}(z) = \Psi(z) {V}^{(\tau)}(z).
\end{equation}
his equation can be represented by the following picture notation:
\vspace{0.2in}
\begin{center}
\begin{tikzpicture}
\draw [ultra thick] (0,0) -- (2,0);
\draw [fill] (2,0) circle [radius=0.1];
\draw [ultra thick] (0.1,0.2) to [out=210,in=150] (0.1,-0.2);
\node at (2.3,0) {$\tau$};
\node at (2.9,0) {$=$};
\draw [ultra thick] (3.4,0) -- (4.4,0);
\draw [ultra thick] (0.1+3.4,0.2) to [out=210,in=150] (0.1+3.4,-0.2);
\draw [very thick] (4.53,0) circle [radius=0.13];
\draw [very thick] (5,0) circle [radius=0.13];
\draw [ultra thick] (5.13,0) -- (6.13,0);
\draw [fill] (6.13,0) circle [radius=0.1];
\node at (6.43,0) {$\tau$};
\end{tikzpicture}
\end{center}

\noindent 2)
The capping operator $\Psi(z)$ is the fundamental solution of the quantum difference equation:
\begin{equation}
\label{difference}
\Psi(q^{\mathcal{L}_i}z)=\mathbf{M}_{\mathcal{L}_i}(z) \Psi(z) \mathcal{L}_i^{-1},
\end{equation}
where $\mathcal{L}_i=\det(V_i)$,
$\mathcal{L}$ is the operator of classical multiplication by the corresponding line bundle and $(q^{\mathcal{L}}z)^{\ard}=q^{\langle\mathcal{L}, \ard\rangle}z^{\ard}$, where $\ard\in H_2(X, \mathbb{Z}), \mathcal{L}\in Pic(X)$.
The matrix $\mathbf{M}_{\mathcal{L}_i}(z)$ is
\begin{equation}
\label{qdeop}
\mathbf{M}_{\mathcal{L}_i}(z)=\sum\limits_{\ard=\vec{0}}^{\infty} z^{\ard} {\rm{ev}}_{\ast}\left(\textsf{QM}^{\ard}_{{\rm{relative}}\, p_1,p_2},\widehat{{\O}}_{{\rm{vir}}}\det H^{\bullet}\left(\mathscr{V}_i\otimes \pi^{\ast}(\mathcal{O}_{p_1})\right)\right)\bold{G}^{-1},
\end{equation}
where $\pi$ is a projection from $\mathcal{C}\to {\mathbb{P}^{1}}$ as in definition \ref{quasimap} and  $\mathcal{O}_{p_1}$ is a class of
point $p_1\in {\mathbb{P}^{1}}$.

\end{theorem}

\noindent {\bf Remark.} The explicit form of operator $\mathbf{M}_{\mathcal{L}_i}$ is known for arbitrary Nakajima variety.
It is constructed in terms of representation theory terms of quantum loop algebra associated with a quiver  \cite{Okounkov:2016sya}.\\

Operators $\mathbf{M}_{\mathcal{L}_i}(z)$ turn out to be closely related to quantum tautological line bundles as the following Theorem suggests, which is a direct generalization of Theorem 10 of \cite{Pushkar:2016qvw}.

\begin{theorem}
In the limit $q\to 1$ operators $\mathbf{M}_{\mathcal{L_i}}(z)$ coincide with the operators of quantum multiplication on the corresponding quantum tautological bundles:
\begin{equation}
\lim_{q\to 1}\mathbf{M}_{\mathcal{L}_i}(z)=\widehat{\mathcal{L}}_i(z).
\end{equation}
\end{theorem}

We will use this fact to compute the formula for the eigenvalues of the operators $\hat{\tau}(z)$.

Let us introduce the following notation.  The
eigenvalues of $\widehat{\mathcal{L}}_i(z)$ are $\lambda_{{\bf p},i}(z)$, so that
$\lambda_{{\bf p},i}(0)=\lambda^0_{{\bf p},i}$, the eigenvalue of the classical multiplication on $\mathcal{L}_i$, corresponding to a fixed point $\mathbf{p} \in X^{\mathsf{T}}$. Using standard arguments of perturbation theory \cite{kato1952}, the above assumption gives:
\begin{lemma} \label{lemmpowe}
{\it The eigenvalues of $\widehat{\mathcal{L}}_i(z)$ are power series in K\"ahler parameters $\lambda_{{\bf p},i}(z) \in \mathbb{C}[[z_1,z_2,\dots]]$. }
\end{lemma}

\begin{proof}
We assume that there is only one K\"ahler parameter which we denote $z$. 
The general case then follow from the same argument applied for each $z_i$. 

The eigenvalues of $\widehat{\mathcal{L}}_i(z)$ belong to the algebraic closure of the field of Laurant series, i.e., they are elements of the field of Puiseux series in $z$. Assume that for some $\widehat{\mathcal{L}}_i(z)$, there is an eigenvalue which is a non-trivial Puiseux series. In other words it is of the form
$$
\lambda_{{\bf p},i}(z) = \lambda^{0}_{{\bf p},i} +  \lambda^{1}_{{\bf p},i} \, z^{1/m}+ \lambda^{2}_{{\bf p},i} \,z^{2/m}+ \dots, \ \ \ m \in \mathbb{N}
$$
with $m>1$. There are no negative powers of $z$ because $\widehat{\mathcal{L}}_i(0)$ are regular by our assumption.

Then, there is a set of $m$ eigenvectors, say, labeled by subset of fixed points $\{ {\bf p}_1,\dots,{\bf p}_m\}$ which undergo a cyclic permutation once we go around $z=0$ along a circle of sufficiently small radius, i.e., when the K\"ahler parameter transforms $z\to z e^{2\pi i }$. 
This is only possible when the leading coefficient of the eigenvalues  $\lambda_{{\bf p}_1,j}(0)=\dots =\lambda_{{\bf p}_m,j}(0)$ for all $j$.
In other words, there is no $\mathcal{L}_j$ for which the corresponding eigenvalues are distinct. We arrive at a contradiction, thus $m=1$. 
\end{proof}

Let ${l}_{\mathbf{p}, i}=\frac{\lambda_{{\bf p},i}(z)}{\lambda^0_{{\bf p},i}}$ be the normalized eigenvalue.
\begin{lemma}
The following function
$$f(t)=\exp\Big({\frac{1}{q-1}\int d_q t\ln\l(t)}\Big),$$ where $\int d_q tf(t)=(1-q)\sum^{\infty}_{n=0}f(tq^n)$ is the standard Jackson $q$-integral, satisfies
 $$f(qt)=l(t)f(t)\,.$$
\end{lemma}
\noindent We denote
\begin{equation}
F_{\mathbf{p}}(z)=\exp\Big({\frac{1}{q-1}
\sum_{i\in I}\int d_q z_i\ln\lambda_{{\bf p},i}(z)}\Big)
\end{equation}
Let us formulate an omnibus theorem concerning the solutions of the system of difference equations and eigenvalues of quantum multiplication operators.
\begin{theorem}
\begin{enumerate}
\item The operator $\Psi(0)$ is the identity operator.

\item Let $\Psi_{{\fp}}(z)$ be the $\fp$-th column of the matrix $\Psi(z)$.
In the limit $q\mapsto 1$ the capping operator has the following asymptotic expansion
\begin{equation}
\Psi_{\it{\fp}}(z) =F_{\mathbf{p}}(z)\Big(\psi_\fp(z) + \cdots\Big),
\end{equation}
where $\psi_{\fp}(z)$ are the column eigenvectors of the operators of quantum multiplication corresponding to the fixed point $\fp$ and dots stand for the terms vanishing in the limit $q\to 1$.

\item The identity element in the quantum K-theory decomposes in the following manner
\begin{equation}
\hat{\bold{1}}(z) = \sum\limits_{\fp} v_{\fp}(z) \psi_{\fp}(z)\,,
\end{equation}
where $v_{\fp}(z)$ are the eigenvalues of quantum multiplication .

\item The coefficients of the bare vertex function  have the following  $q\to 1$ asymptotic
in the fixed points basis
\begin{eqnarray}
V^{(\tau)}_{\fp}(z)=F_{\mathbf{p}}(z)^{-1}(\tau_{\fp}(z) v_\fp(z) +\cdots),
\end{eqnarray}
where $\tau_{\fp}(z)$ denotes the eigenvalue of the operator of quantum multiplication by quantum tautological bundle $\hat{\tau}(z)$ for the eigenvector $\psi_{\fp}(z)$, dots stand for the terms vanishing in the limit $q\to 1$.
\end{enumerate}
\end{theorem}
For the proof of this theorem we will refer the reader to \cite{Pushkar:2016qvw}, where it is proven in the case of a single variable $z$, when $X$ is $T^*Gr(k,n)$.
Current theorem is a direct generalization.

Part (4) of the Theorem above immediately implies that the eigenvalues of the operator of quantum multiplication by $\hat{\tau}(z)$ can be computed from the
asymptotics of the bare vertex functions.

\begin{corollary}
The following expression:
\begin{equation}
\label{eval}
\tau_{\fp}(z)=\lim\limits_{q \rightarrow 1 } \dfrac{V^{(\tau)}_{\fp}(z)}{V^{(1)}_{\fp}(z)}
\end{equation}
gives the eigenvalues of the operator of quantum multiplication by $\hat{\tau}(z)$ corresponding to
a fixed point $\fp \in X^{\mathsf{T}}$.
\end{corollary}

\section{Computations for Partial Flags}\label{Sec:QuantumKPartialFlags}
In this section we will study in detail and apply the formalism which we have developed in the previous section to the case when Nakajima quiver variety $X$ is the cotangent bundle to the space of partial flags. In other words, we are interested in studying quantum K-theory of the following quiver of type $A_n$\footnote{We are using standard quaternionic notations.}

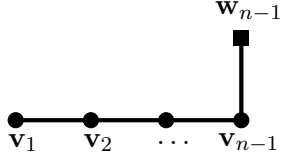
\vspace{0.1in}
\begin{center}
\begin{tikzpicture}
\draw [ultra thick] (0,0) -- (3,0);
\draw [ultra thick] (3,1) -- (3,0);
\draw [fill] (0,0) circle [radius=0.1];
\draw [fill] (1,0) circle [radius=0.1];
\draw [fill] (2,0) circle [radius=0.1];
\draw [fill] (3,0) circle [radius=0.1];
\node (1) at (0.1,-0.3) {$\mathbf{v}_{1}$};
\node (2) at (1.1,-0.3) {$\mathbf{v}_2$};
\node (3) at (2.1,-0.3) {$\ldots$};
\node (4) at (3.1,-0.3) {$\mathbf{v}_{n-1}$};
\fill [ultra thick] (3-0.1,1) rectangle (3.1,1.2);
\node (5) at (3.1,1.45) {$\mathbf{w}_{n-1}$};
\end{tikzpicture}
\end{center}
\vspace{0.1in}

The stability condition is chosen so that maps $W_{n-1}\to {\bf V}_{n-1}$ and $V_i\to V_{i-1}$ are surjective. For the variety to be non-empty the sequence $\mathbf{v}_{1},\ldots ,\mathbf{v}_{n-1}, \mathbf{w}_{n-1}$ must be non-decreasing. The fixed points of this Nakajima quiver variety and the stability condition are classified by chains of subspaces spanned by coordinate vectors $\mathbf{V}_{1}\subset \ldots \subset \mathbf{V}_{n-1}\subset \mathbf{W}_{n-1}$, where $|\mathbf{V}_{i}|=\mathbf{v}_{i}, \mathbf{W}_{n-1}=\{a_1,\ldots , a_{\mathbf{w}_{n-1}} \}$. The special case when ${\bf v}_i=i$, ${\bf w}_{n-1}=n$ is known as complete flag variety, which we denote as $\mathbb{F}l_n$.
It will be convenient to introduce the following notation: $\mathbf{v}_i^{\prime}=\mathbf{v}_{i+1}-\mathbf{v}_{i-1}$, for $i=2,\ldots ,n-2$, $\mathbf{v}_{n-1}^{\prime}=\mathbf{w}_{n-1}-\mathbf{v}_{n-2}$, $\mathbf{v}_1^{\prime}=\mathbf{v}_{2}$.\\

\noindent{\bf Remark.} In principle, in the computations below one could add extra framings to vertices to study the most generic situation in the setting of $A_n$ quiver, but we shall refrain from doing it in this work to make calculations more transparent and simple.

\subsection{Bare vertex for partial flags}
\label{Subsec:BareVertex}
The key for computing the bare vertex is the localization theorem in K-theory,
which gives the following formula for the equivariant pushforward, which constitutes bare vertex $V^{(\tau)}_{\fp}(z)$:
$$
V^{(\tau)}_{\fp}(z)=\sum\limits_{{\bf d}\in \mathbb{Z}^n_{\ge0} }\sum\limits_{(\mathscr{V},\mathscr{W}) \in  (\textsf{QM}^{\bf d}_{{\rm nonsing} \, p_2})^{\mathsf{T}}}\, \hat{s}(  \chi({\bf d}) )\, z^{\bf d} q^{\deg(\mathscr{P})/2} \tau (\left.\mathscr{V}\right|_{p_1}).
$$
Here the sum runs over the $\mathsf{T}$-fixed quasimaps which take value $\fp$ at the nonsingular point $p_2$. We use notation $\hat{s}$ for the Okounkov's roof function defined by
$$
\hat{s}(x)=\dfrac{1}{x^{1/2}-x^{-1/2}}, \ \ \ \hat{s}(x+y)=\hat{s}(x)\hat{s}(y).
$$ and it is applied to the virtual tangent bundle:
\beq \label{pcontr}
\chi(\textbf{d})={\rm char}_{\mathsf{T}}\Big( T^{vir}_{\{(\mathscr{V}_i\},\ \mathscr{W}_{n-1})} \textsf{QM}^{\bf d}  \Big).
\eeq

The condition ${\bf d}\in \mathbb{Z}^n_{\ge0} $ is determined by stability conditions, which characterize all allowed degrees for quasimaps.

\noindent 

It will be convenient to adopt the following notations:
\beq
\varphi(x)=\prod^{\infty}_{i=0}(1-q^ix),\quad  \{x\}_{d}=\dfrac{(\hbar/x,q)_{d}}{(q/x,q)_{d}} \, (-q^{1/2} \hbar^{-1/2})^d, \ \ \textrm{where}  \ \ (x,q)_{d}=\frac{\varphi(x)}{\varphi(q^dx)}.\nonumber
\eeq
The following statement is true (for the proof see section 3.4 of \cite{Pushkar:2016qvw}).

\begin{lemma}
\label{Th:LemmaEquivBundleContrib}
The contribution of equivariant line bundle $xq^{-d}\mathcal{O}(d)\subset \mathcal{P}$ to $\chi(\mathbf{d})$ is $\{x\}_d$.
\end{lemma}

To compute the vertex function we will also need to classify fixed points of $\textsf{QM}^{\bf d}_{{\rm nonsing} \, p_2}$. Such a point is described by the data $(\{\mathscr{V}_i\},\{\mathscr{W}_{n-1}\})$, where ${\rm deg}\mathscr{V}_i=d_i, {\rm deg}\mathscr{W}_{n-1}=0$. Each bundle $\mathscr{V}_i$ can be decomposed into a sum of line bundles $\mathscr{V}_i=\mathcal{O} (d_{i,1}) \oplus \ldots \oplus \mathcal{O}(d_{i,{\bf v}_i})$ (here $d_i=d_{i,1}+\ldots +d_{i,{\bf v}_i}$). For a stable quasimap with such data to exist the collection of $d_{i,j}$ must satisfy the following conditions
\begin{itemize}
\item $d_{i,j}\geq 0$,
\item for each $i=1,\ldots ,n-2$ there should exist a subset in $\{d_{i+1,1},\ldots d_{i+1,{\bf v}_{i+1}}\}$ of cardinality ${\bf v}_i$ $\{d_{i+1,j_1},\ldots d_{i+1,j_{{\bf v}_i}}\}$, such that $d_{i,k}\geq d_{i+1, j_k}$.
\end{itemize}
To summarize, we will denote collections satisfying such conditions as lying in a chamber $d_{i,j}\in C$.

Now we are ready to sum up contributions for the entire vertex function.

\begin{proposition}
Let $\fp=\mathbf{V}_{1}\subset \ldots \subset \mathbf{V}_{n-1}\subset \{a_1,\cdots,a_{\mathbf{w}_{n-1}}\}$ $(\mathbf{V}_{i}=\{ x_{i,1},\ldots x_{i,\mathbf{v}_{i}}\})$ be a chain of subspaces defining a torus fixed point $\fp\in X^{\mathsf{T}}$.
Then the coefficient of the vertex function for this point is given by:
$$
V^{(\tau)}_{\fp}(z) = \sum\limits_{d_{i,j}\in C}\, z^{\ard} q^{N(\ard)/2}\, EHG\ \ \tau(x_{i,j} q^{-d_{i,j}}),
$$
where $\ard=(d_1,\ldots ,d_{n-1}),d_i=\sum_{j=1}^{\mathbf{v}_i}d_{i,j},N(\ard)=\mathbf{v}_i^{\prime}d_i$,
$$
E=\prod_{i=1}^{n-1}\prod\limits_{j,k=1}^{\mathbf{v}_i}\{x_{i,j}/x_{i,k}\}^{-1}_{d_{i,j}-d_{i,k}},
$$
$$
G=\prod\limits_{j=1}^{\mathbf{v}_{n-1}} \prod\limits_{k=1}^{\mathbf{w}_{n-1}} \{x_{n-1,j}/a_k\}_{d_{n-1,j}},
$$
$$
H=\prod_{i=1}^{n-2}\prod_{j=1}^{\mathbf{v}_i}\prod_{k=1}^{\mathbf{v}_{i+1}}\{x_{i,j}/x_{i+1,k}\}_{d_{i,j}-d_{i+1,k}}.
$$
\end{proposition}
\begin{proof}
For the proof we need to gather all contributions $\mathscr{P}$, which separate into 3 types:
$$
\mathscr{P}=\mathscr{W}_{n-1}^{\ast}\otimes \mathscr{V}_{n-1}+\sum_{i=1}^{n-2}\mathscr{V}_{i+1}^{\ast}\otimes \mathscr{V}_i-\sum_{i=1}^{n-1}\mathscr{V}_{i}^{\ast}\otimes \mathscr{V}_i
$$
so that their input in the localization formula is as follows.
$$
\mathscr{W}_{n-1}^{\ast}\otimes \mathscr{V}_{n-1}\longrightarrow \prod\limits_{j=1}^{\mathbf{v}_{n-1}} \prod\limits_{k=1}^{\mathbf{w}_{n-1}} \{x_{n-1,j}/a_k\}_{d_{n-1,j}},
$$
$$
\mathscr{V}_{i+1}^{\ast}\otimes \mathscr{V}_i \longrightarrow \prod_{j=1}^{\mathbf{v}_i}\prod_{k=1}^{\mathbf{v}_{i+1}}\{x_{i,j}/x_{i+1,k}\}_{d_{i,j}-d_{i+1,k}},
$$
$$
\mathscr{V}_{i}^{\ast}\otimes \mathscr{V}_i \longrightarrow \prod_{i=1}^{n-1}\prod\limits_{j,k=1}^{\mathbf{v}_i}\{x_{i,j}/x_{i,k}\}^{-1}_{d_{i,j}-d_{i,k}},
$$
Note, that $\deg(\mathscr{P}) =N(\ard)$. That gives the polarization term $q^{N(\ard)/2}$ in the vertex.
\end{proof}
The same formula for the vertex can an be obtained using the following integral representation \cite{Aganagic:2017la}, \cite{Aganagic:2017be}. It is very useful for a lot of applications, in particular for computing the eigenvalues $\tau_{\bf p}(z)$.
\begin{proposition}
The bare vertex function is given by
\beq \label{vertexint}
V^{(\tau)}_{\fp}(z)=
 \dfrac{1}{2 \pi i \mathsf{a}_{\fp}} \int\limits_{C_{\bf p}} \, \prod\limits_{i=1}^{n-1}\prod\limits_{j=1}^{\mathbf{v}_i}e^{-\frac{\ln(z^{\sharp}_i) \ln(s_{i,j}) }{\ln(q)}} \, E_{\rm{int}} G_{\rm{int}} H_{\rm{int}} \tau(s_1,\cdots,s_k)\prod\limits_{i=1}^{n-1}\prod\limits_{j=1}^{\mathbf{v}_i} \dfrac{d s_{i,j}}{s_{i,j}} ,
\eeq
where
$$
E_{\rm{int}}=\prod_{i=1}^{n-1}\prod\limits_{j,k=1}^{\mathbf{v}_i} \dfrac{\varphi\Big( \frac{s_{i,j}}{ s_{i,k}}\Big)}{\varphi\Big(\frac{q}{\hbar} \frac{s_{i,j}}{ s_{i,k}}\Big)},
$$
$$
G_{\rm{int}}=\prod\limits_{j=1}^{\mathbf{w}_{n-1}}\prod\limits_{k=1}^{\mathbf{v}_{n-1}} \dfrac{\varphi\Big(\frac{q}{\hbar} \frac{s_{n-1,k}}{ a_j }\Big)}{\varphi\Big(\frac{s_{n-1,k}}{ a_j }\Big)},
$$
$$
H_{\rm{int}}=\prod_{i=1}^{n-2}\prod\limits_{j=1}^{\mathbf{v}_{i+1}}\prod\limits_{k=1}^{\mathbf{v}_{i}} \dfrac{\varphi\Big(\frac{q}{\hbar} \frac{s_{i,k}}{ s_{i+1,j} }\Big)}{\varphi\Big(\frac{s_{i,k}}{ s_{i+1,j} }\Big)},
$$
$$\mathsf{a}_{\fp}=\prod\limits_{i=1}^{n-1}\prod\limits_{j=1}^{\mathbf{v}_i}e^{-\frac{\ln(z^{\sharp}_i) \ln(s_{i,j}) }{\ln(q)}} \, E_{\rm{int}} G_{\rm{int}} H_{\rm{int}}\Big|_{s_{i,j}=x_{i,j}},$$
and the contour $C_{\bf p}$ runs around points corresponding to chamber $C$, which we have defined after Lemma \ref{Th:LemmaEquivBundleContrib}, and the shifted variable  $z^{\sharp}=z(-\hbar^{^1/_2})^{\det(\mathscr{P})}$
\footnote{
Note that here we are using the notation defined for $z$ for $(-\hbar^{^1/_2})$, i.e.
$$
z^{\sharp}=\prod_{i=1}^{n-1}z^{\sharp}_i,
$$
$$
z^{\sharp}_i=z_i(-\hbar^{^1/_2})^{\mathbf{v}_i^{\prime}}.
$$}.
\end{proposition}
\subsection{Bethe Equations and Baxter Operators}\label{Subsec:BEandBO}
We are now ready to compute the eigenvalues of the operators corresponding to the tautological bundles.

\begin{theorem}\label{betheth}The eigenvalues of
$\hat{\tau}(z)\circledast$ is given by $\tau({s_{i,k}})$, where $s_{i,k}$ satify
Bethe equations:
\begin{align}\label{eq:XXZBetheGeom}
\prod_{j=1}^{\mathbf{v}_{2}}\frac{s_{1,k}-s_{2,j}}{s_{1,k}-\hbar s_{2,j}}&=z_{1}{(-{\hbar}^{^1/_2})}^{-{\mathbf{v}_{1}^{\prime}}}\prod_{\substack{j=1 \\ j\neq k}}^{\mathbf{v}_{1}}\frac{s_{1,j}-s_{1,k}\hbar}{s_{1,j}\hbar-s_{1,k}}\,,\cr
\prod_{j=1}^{\mathbf{v}_{i+1}}\frac{s_{i,k}-s_{i+1,j}}{s_{i,k}-\hbar s_{i+1,j}}\prod_{j=1}^{\mathbf{v}_{i-1}}\frac{s_{i-1,j}-\hbar s_{i,k}}{s_{i-1,j}-s_{i,k}}&=z_{i}{(-{\hbar}^{^1/_2})}^{-{\mathbf{v}_{i}^{\prime}}}\prod_{\substack{j=1 \\ j\neq k}}^{\mathbf{v}_{i}}\frac{s_{i,j}-s_{i,k}\hbar}{s_{i,j}\hbar-s_{i,k}}\,,\\
\prod_{j=1}^{\mathbf{w}_{n-1}}\frac{s_{n-1,k}-a_{j}}{s_{n-1,k}-\hbar a_{j}}\prod_{j=1}^{\mathbf{v}_{n-2}}\frac{s_{n-2,j}-\hbar s_{n-1,k}}{s_{n-2,j}-s_{n-1,k}}&=z_{n-1}{(-\hbar^{^1/_2})}^{-{\mathbf{v}_{n-1}^{\prime}}}\prod_{\substack{j=1 \\ j\neq k}}^{\mathbf{v}_{n-1}}\frac{s_{n-1,j}-s_{n-1,k}\hbar}{s_{n-1,j}\hbar-s_{n-1,k}}\notag\,,
\end{align}
where $k=1,\dots, \textbf{v}_i$ for $i=1,\dots, \textbf{v}_{n-1}$.
\end{theorem}
\begin{proof}
There are several ways of obtaining these equations. One way corresponds to the study of asymptotics of \eqref{eval} as it was done in section 3.5 of \cite{Pushkar:2016qvw}. However, there is a shortcut recently provided by \cite{Aganagic:2017be}. One chooses a preimage of the class $TX$ in $K_{\prod_iGL(V_i)\times GL(W_{n-1})}(pt)$ under the Kirwan map, so that  $a_j$ are coordinates of the torus acting on $W_{n-1}$ and $s_{i,k}$ are coordinates of the torus acting on $V_i$. In this case we have
\begin{eqnarray}
&&TX=T(T^{*}\text{Rep}(\mathbf{v},\mathbf{w}))-\sum_{i\in I}(1+\hbar)\text{End}(V_i)=\\
&&\sum_{i=1}^{n-2}\sum_{k=1}^{\mathbf{v}_i}\sum_{j=1}^{\mathbf{v}_{i+1}}\left(\frac{s_{i,k}}{s_{i+1,j}}+\frac{s_{i+1,j}\hbar}{s_{i,k}}\right)+\sum_{k=1}^{\mathbf{v}_{n-1}}\sum_{j=1}^{\mathbf{w}_{n-1}}\left(\frac{s_{n-1,k}}{a_j}+\frac{a_j\hbar}{s_{n-1,k}}\right)-(1+\hbar)\sum_{i\in I}\sum_{j,k=1}^{\mathbf{v}_i}\frac{s_{i,j}}{s_{i,k}}.\nonumber
\end{eqnarray}
To get Bethe equations we use the following formula from the Appendix to \cite{Aganagic:2017be}:
$$
\widehat{a}\left(s_{i,k}\frac{\partial}{\partial s_{i,k}}TX\right)=z_i,
$$
where $\widehat{a}\left(\sum n_i x_i\right)=\prod {\left(x_i^{1/2}-x_i^{-1/2}\right)}^{n_i}$.
\end{proof}


The equations (\ref{eq:XXZBetheGeom}) are Bethe ansatz equations for the periodic anisotropic $\mathfrak{gl}(n)$ XXZ spin chain on $\mathbf{w}_{n-1}$ sites with twist parameters $z_1,\dots, z_{n-1}$, impurities (shifts of spectral parameters) $a_1,\dots,a_{\mathbf{w}_{n-1}}$, and quantum parameter $\hbar$, see e.g. \cite{Bogolyubov:1993qism}, \cite{Reshetikhin:2010si}.\\

Let us consider the quantum tautological bundles $\widehat{\Lambda^kV_i}(z)$, $k=1,\dots, \mathbf{v}_i$. It is useful to construct a generating function for
those, namely
\begin{equation}
\mathbf{Q}_i(u)=\sum^{\mathbf{v}_i}_{k=0}(-1)^k u^{\mathbf{v}_i-k}\hbar^{\frac{ik}{2}}\widehat{\Lambda^kV_i}(z).
\end{equation}
The seemingly strange $\hbar$ weights will be necessary in Section 4. In the integrable system literature these operators are known as Baxter operators \cite{Baxter:1982zz},\cite{Reshetikhin:2010si}.
The following Theorem is a consequence of (\ref{eval}).
\begin{proposition}
The eigenvalues of the operator ${\bf Q}_i(u)$ are the following polynomials in $u$:
\begin{equation}\label{baxteig}
Q_i(u)=\prod^{\bf{v}_i}_{k=1}(u-\hbar^{\frac{i}{2}}s_{i,k}),
\end{equation}
so that the coefficients are elementary symmetric functions in  $s_{i,k}$ for
fixed i.
\end{proposition}

\noindent {\bf Remark.}
To obtain the full Hilbert space of a
$\mathfrak{gl}(n)$ $XXZ$ model one has to consider a disjoint union of all partial flag varieties with framing ${W}_{n-1}$ fixed, so that in the basis of fixed points the classical equivariant K-theory can be expressed as a tensor product $\mathbb{C}^n(a_1)\otimes\mathbb{C}^n(a_2)\otimes \dots   \mathbb{C}^n(a_{{\bf w}_{n-1}})$,
where each of $\mathbb{C}^n(a_i)$ is an evaluation representation of $U_{\hbar}(\widehat{\mathfrak{gl}}(n))$, see e.g. \cite{Nakajima:2001qg}.
There is a special interesting question regarding universal formulas for operators
$\mathbf{Q}_i(u)$ which we used in \cite{Pushkar:2016qvw}
for $\mathfrak{gl}(2)$ model, corresponding to prefundamental representations of the Borel subalgebra of
$U_{\hbar}(\widehat{\mathfrak{gl}}(n))$ \cite{Frenkel:2013uda}.

\subsection{Compact limit}\label{Sec:Compact}

In this section we study the limit of $\hbar\to \infty$ of the vertex functions and the Bethe equations.
We recall that $\hbar$ is the equivariant parameter of the torus which scales the cotangent directions in $T^{*}Fl_n$. Later in Section \ref{Sec:tRSToda} we also show that the limit $\hbar\to \infty$ of the ring $QK_{T}(T^{*}Fl_n)$ coincides with the quantum K-theory ring of complete flag varieties computed by Givental and Lie in \cite{2001math8105G}. It is thus natural to call   $\hbar\to \infty$ - the compact limit.

In order to make the limit $\hbar\to \infty$ of vertex functions and Bethe equations well defined, we also need to rescale the equivariant parameters $z_i$ by powers of $\hbar$. Here is the exact procedure:

\begin{theorem}
In the limit  $\hbar\to \infty$ with K\"ahler parameters $\{z\}$ scaled so that  $\{z^\sharp\}$ remains fixed

\label{betheth2}\begin{enumerate}
\item

the vertex functions \eqref{vertexint} have well defined limits. \\
\item
The limits of Bethe ansatz equations exist and take the form
\begin{align}\label{eq:XXZCompact}
\prod_{j=1}^{\mathbf{w}_{n-1}}\frac{s_{n-1,k}-a_{j}}{a_{j}}\prod_{j=1}^{\mathbf{v}_{n-2}}\frac{s_{n-1,k}}{s_{n-2,j}-s_{n-1,k}}&={z^\sharp}_{n-1}\prod\limits_{{j=1},{j\neq k}}^{\mathbf{v}_{n-1}}\frac{-s_{n-1,k}}{s_{n-1,j}}\,,\quad k=1,\ldots,\mathbf{v}_{n-1}\,,\cr
\prod_{j=1}^{\mathbf{v}_{i+1}}\frac{s_{i,k}-s_{i+1,j}}{s_{i+1,j}}\prod_{j=1}^{\mathbf{v}_{i-1}}\frac{s_{i,k}}{s_{i-1,j}-s_{i,k}}&={z^{\sharp}}_{i}\prod\limits_{{j=1},{j\neq k}}^{\mathbf{v}_{i}}\frac{-s_{i,k}}{s_{i,j}}\,,\quad k=1,\ldots,\mathbf{v}_{i}\,,\quad i=2,\ldots, n-2\,,\cr
\prod_{j=1}^{\mathbf{v}_{2}}\frac{s_{1,k}-s_{2,j}}{s_{2,j}}&={z^{\sharp}}_{1}\prod\limits_{{j=1},{j\neq k}}^{\mathbf{v}_{1}}\frac{-s_{1,k}}{s_{1,j}}\,,\quad k=1,\ldots,\mathbf{v}_{1}\,.
\end{align}
\end{enumerate}
\end{theorem}
\begin{proof}
When applying the localization theorem to compute the bare vertex for the cotangent bundle to partial flags we can break up the terms in pairs of the form $(\omega,\omega^{-1}\hbar )$. The latter corresponds to the cotangent fiber. The contribution of such a pair to the vertex will be equal to:
$$
\frac{1}{\omega^{^1/_2}-\omega^{-^1/_2}}\frac{1}{(\hbar\omega^{-1})^{^1/_2}-(\hbar\omega^{-1})^{-^1/_2}}=\frac{1}{1-\omega^{-1}}\frac{-\hbar^{^1/_2}}{1-\hbar^{-1}\omega^{-1}}.
$$
Therefore after rescaling by $(-\hbar^{^1/_2})$, which corresponds to expressing $z$ in terms of $z^{\sharp}$ will be equal to $\frac{1}{1-\omega^{-1}}$ in the $\hbar\to\infty$ limit, that is exactly the contribution of $\omega$ in the case of the partial flag variety. One can check that the resulting sum is indeed finite by looking at the integral formula for the vertex (\ref{vertexint}). Namely, the integrand in the expression for the vertex after fiber removal is as follows:
$$
E_{\rm{int}}\to \prod_{i=1}^{n-1}\prod\limits_{j,k=1}^{\mathbf{v}_i} \varphi\Big( \frac{s_{i,j}}{ s_{i,k}}\Big),
$$
$$
G_{\rm{int}}\to \prod\limits_{j=1}^{\mathbf{w}_{n-1}}\prod\limits_{k=1}^{\mathbf{v}_{n-1}} \frac{1}{\varphi\Big(\frac{s_{n-1,k}}{ a_j }\Big)},
$$
$$
H_{\rm{int}}\to \prod_{i=1}^{n-2}\prod\limits_{j=1}^{\mathbf{v}_{i+1}}\prod\limits_{k=1}^{\mathbf{v}_{i}} \dfrac{1}{\varphi\Big(\frac{s_{i,k}}{ s_{i+1,j} }\Big)}.
$$
In order to obtain the corresponding Bethe equations, one can again compute $q\to 1$ asymptotics or just simply  evaluating the limit $\hbar \to \infty$ of  (\ref{eq:XXZBetheGeom}) while expressing $z$ in terms of $z_{\sharp}$.
\end{proof}

\section{The XXZ/tRS Duality}\label{Sec:XXZ/tRS}
In this section we discuss the duality between XXZ spin chain and trigonometric Ruijsenaars-Schneider (tRS) model which first appeared in physics literature \cite{Gaiotto:2013bwa, Bullimore:2015fr}. It was there referred to as \textit{quantum/classical duality}. Here we will prove main statements of \cite{Bullimore:2015fr}.


\subsection{The XXZ Spin Chain}\label{Subsec:XXZ}
To start let us change the K\"ahler parameters in Bethe equations \eqref{eq:XXZBetheGeom} according to
\begin{align}\label{eq:zvdzeta}
z_1&=\frac{\zeta_{1}}{\zeta_2}\,,\cr
z_i&=\frac{\zeta_{i}}{\zeta_{i+1}}\,,\quad i=2,\dots, n-2\cr
z_{n-1}&=\frac{\zeta_{n-1}}{\zeta_{n}}\,.
\end{align}
In what follows we shall treat K\"ahler variables $\zeta_i$ as formal.

Additionally after rescaling Bethe roots and equivariant parameters
\begin{equation}
\sigma_{i,k}=\hbar^{\frac{i}{2}}s_{i,k}\,,\quad i=1,\dots, n-1\,,\qquad \mathsf{a}_k = \hbar^{\frac{n}{2}}a_k\,,
\label{eq:rescaling}
\end{equation}
we arrive at the following set of equations which is equivalent to \eqref{eq:XXZBetheGeom}
\begin{align}
\label{eq:BetheEqDetailed}
\frac{\zeta_{1}}{\zeta_{2}}\cdot \prod_{ \beta \neq \alpha}^{\textbf{v}_1}\frac{\hbar\sigma_{1,\alpha} - \sigma_{1,\beta}}{\hbar \sigma_{1,\beta}-\sigma_{1,\alpha}} \cdot \prod_{\beta=1}^{\textbf{v}_{2}}\frac{\sigma_{1,\alpha} - \hbar^{^1/_2} \sigma_{2,\beta}}{ \sigma_{2,\beta}-\hbar^{^1/_2}\sigma_{1,\alpha}}&=(-1)^{\delta_1}\,,\notag\\
\frac{\zeta_{i}}{\zeta_{i+1}}\cdot\prod_{\beta=1}^{\textbf{v}_{i-1}}\frac{ \sigma_{i,\alpha}- \hbar^{^1/_2}\sigma_{i-1,\beta}}{\sigma_{i-1,\beta}-\hbar^{^1/_2} \sigma_{i,\alpha}} \cdot \prod_{ \beta \neq \alpha}^{\textbf{v}_{i}}\frac{\hbar \sigma_{i,\alpha} - \sigma_{i,\beta}}{\hbar \sigma_{i,\beta}-\sigma_{i,\alpha}} \cdot \prod_{\beta=1}^{\textbf{v}_{i+1}}\frac{\sigma_{i,\alpha} - \hbar^{^1/_2} \sigma_{i+1,\beta}}{\sigma_{i+1,\beta}-\hbar^{^1/_2} \sigma_{i,\alpha}}&=(-1)^{\delta_i}\,,\\
\frac{\zeta_{n-1}}{\zeta_{n}}\cdot\prod_{\beta=1}^{\textbf{v}_{n-2}}\frac{ \sigma_{n-1,\alpha}- \hbar^{^1/_2}\sigma_{n-2,\beta}}{ \sigma_{n-2,\beta}-\hbar^{^1/_2}\sigma_{n-1,\alpha}} \cdot \prod_{ \beta \neq \alpha}^{\textbf{v}_{n-1}}\frac{\hbar\sigma_{n-1,\alpha} -  \sigma_{n-1,\beta}}{\hbar \sigma_{n-1,\beta}- \sigma_{n-1,\alpha}} \cdot \prod_{\beta=1}^{\textbf{w}_{n-1}}\frac{ \sigma_{n-1,\alpha} - \hbar^{^1/_2}\mathsf{a}_{\beta}}{ \mathsf{a}_{\beta}-\hbar^{^1/_2}\sigma_{n-1,\alpha}}&=(-1)^{\delta_{n-1}}\,,\notag
\end{align}
where in the middle equation $i=2,\dots, n-2$ and $\delta_i =\textbf{v}_{i-1}+\textbf{v}_{i}+\textbf{v}_{i+1}-1$.  The reader may notice that we use slightly non-standard notation for Bethe equations, in particular, parameters $a_\beta$ appear in the last equation $i=n-1$ (instead of the first equation). Sign factors $(-1)^{\delta_i}$ in the right hand sides are artifacts of this choice. However, as we saw in the previous section this way of writing the equations is more convenient from geometric point of view.
Later we shall see that this framework will be convenient in the derivation of the Lax matrix of the trigonometric Ruijsenaars-Schneider model.

Meanwhile, if we denote $\textbf{v}_0=0\,,\textbf{v}_{n}=\textbf{w}_{n-1}$, $\sigma_{n,\beta}=\mathsf{a}_\beta$ for $\beta=1,\dots, \textbf{w}_{n-1}$ then \eqref{eq:BetheEqDetailed} can be written more uniformly as follows
\begin{equation}
\frac{\zeta_{i}}{\zeta_{i+1}}\prod_{\beta=1}^{\textbf{v}_{i-1}}\frac{\sigma_{i,\alpha}-\hbar^{^1/_2} \sigma_{i-1,\beta}}{ \sigma_{i-1,\beta}-\hbar^{^1/_2}\sigma_{i,\alpha}} \cdot \prod_{ \beta \neq \alpha}^{\textbf{v}_{i}}\frac{\hbar\sigma_{i,\alpha} -  \sigma_{i,\beta}}{\hbar\sigma_{i,\beta}- \sigma_{i,\alpha}} \cdot \prod_{\beta=1}^{\textbf{v}_{i+1}}\frac{ \sigma_{i,\alpha} - \hbar^{^1/_2}\sigma_{i+1,\beta}}{ \sigma_{i+1,\beta}-\hbar^{^1/_2}\sigma_{i,\alpha}}=(-1)^{\delta_i}\,.
\label{eq:XXZGen}
\end{equation}


Following (\ref{baxteig}) let us write eigenvalues $Q_i(u)$ of Baxter operators in terms of the new variables and complement it with $Q_n(u)$, being the generating function for elementary symmetric functions of equivariant parameters.
\begin{equation}
Q_i(u) = \prod_{\alpha=1}^{\textbf{v}_i}\left(u - \sigma_{i,\alpha}\right)\,,\qquad P(u) = Q_n(u)= \prod_{\alpha=1}^{\textbf{w}_{n-1}}(u - \mathsf{a}_\alpha)\,.
\label{eq:QPBaxter}
\end{equation}
In addition, we define shifted polynomials when their arguments are multiplied by $\hbar^{-\frac{1}{2}}$ to the corresponding power:
 $Q_i^{(n)}(u)=Q_i(\hbar^{-\frac{n}{2}} u)$, etc.

Then Bethe equations \eqref{eq:XXZGen} can be expressed in terms of these polynomials as follows
\begin{lemma}
The equation for Bethe root $\sigma_{i,\alpha}$ in \eqref{eq:XXZGen} 
arises by setting $u = \sigma_{i,\alpha}$
in the equation below
\begin{equation}
\hbar^{\frac{\Delta_i}{2}} \frac{\zeta_{i}}{\zeta_{i+1} } \frac{Q_{i-1}^{(1)} Q_{i}^{(-2)} Q_{i+1}^{(1)}}{Q_{i-1}^{(-1)} Q_{i}^{(2)} Q_{i+1}^{(-1)}} = -1\,,
\label{eq:TQBaxter}
\end{equation}
where $\Delta_i = \textbf{v}_{i+1}+\textbf{v}_{i-1}-2\textbf{v}_{i}$.
\end{lemma}
Note that sign $\delta_i$ disappeared.

In order to proceed further we need to rewrite \eqref{eq:TQBaxter} in a slightly different way.

\begin{proposition}\label{eq:PropQQtilde}
Suppose that $\zeta_{i+1}\notin \hbar^{\mathbb{N}} \zeta_{i}$ for all $i$. Then the system of equations \eqref{eq:TQBaxter} upon change of parameters $\widetilde \zeta_i = \zeta_i \, \hbar^{-\sum_{j=1}^{i-1}\frac{\Delta_j}{2}}$ is equivalent to the existence of auxiliary polynomials $\widetilde Q_i$ of degrees $\textbf{v}_{i-1}-\textbf{v}_{i}+\textbf{v}_{i+1}$ satisfying the following system of equations
\begin{equation}
\widetilde \zeta_{i+1} Q^{(1)}_i \widetilde Q^{(-1)}_i - \widetilde \zeta_{i} Q^{(-1)}_i \widetilde Q^{(1)}_i = (\widetilde \zeta_{i+1}-\widetilde \zeta_i)Q_{i-1}Q_{i+1}\, ,
\label{eq:QQrelations}
\end{equation}
The polynomials $\widetilde Q_i(u)$ are unique.
\end{proposition}

\begin{proof}  
Let $g(z)=\widetilde Q_i(z)/Q_i(z)$ and $f(z)=(\widetilde\zeta_{i+1}-\widetilde\zeta_{i})Q_{i-1}^{(1)}Q_{i+1}^{(1)}$ so that \eqref{eq:QQrelations} can be written as
\begin{equation}\label{e:gf}
\widetilde\zeta_{i+1}g_i(z)-\widetilde\zeta_i g^{(2)}_i(z)=\frac{f(z)}{Q_{i}(z)Q^{(2)}_i(z)}.
\end{equation}
Then we have the following partial fraction decompositions
\begin{equation} \begin{aligned}
&\frac{f(z)}{Q_{i}(z)Q^{(2)}_i(z)}=h(z)-\sum_a\frac{b_a}{z-\sigma_{i,a}}+\sum_a\frac{c_a}{\hbar^{-1} z-\sigma_{i,a}},\\
&g_i(z)=\tilde{g}_i(z)+\sum_a\frac{d_a}{z-\sigma_{i,a}}
\end{aligned}
\end{equation}
where $h(z)$ and $\tilde{g}_i(z)$ are polynomials.  In order for the
residues at each $\sigma_{i,a}$ to match on both sides of \eqref{e:gf},
one needs
\begin{equation}d_a=\frac{b_a}{\widetilde\zeta_{i}}=\frac{c_a}{\widetilde\zeta_{i+1}}.
\end{equation}
The second equality is merely the Bethe equations \eqref{eq:TQBaxter} in
the alternate form
\begin{equation}
 \text{Res}_{\sigma_{i,a}}\left[\frac{f(z)}{\widetilde\zeta_{i+1} Q_i(z)Q^{(2)}_i(z)}\right]+ \text{Res}_{\sigma_{i,a}}\left[\frac{f^{(-2)}(z)}{\widetilde\zeta_{i} Q^{(-2)}_i(z)Q_i(z)}\right]=0\,,
\end{equation}
or, equivalently,
\begin{equation}
\left(\left.\frac{Q^{(1)}_{i-1}Q^{(1)}_{i+1}}{\widetilde\zeta_{i+1}
    Q^{(2)}_i}+\frac{Q^{(-1)}_{i-1}Q^{(-1)}_{i+1}}{\widetilde\zeta_{i}Q^{(-2)}_i}\right)\right|_{\sigma_{i,a}}=0.
\end{equation} 
Next, to solve for the polynomial $\tilde{g}_i(z)$, set
$\tilde{g}_i(z)=\sum r_jz^j$ and $h(z)=\sum s_j z^j$.  We then obtain
the equations $r_j(\widetilde\zeta_{i+1}\hbar^{-j} -\widetilde\zeta_i )=s_j$.  Our assumptions
on the $\zeta$'s imply that these equations are always solvable.
Thus, there exist polynomials $\widetilde{Q}_i(z)$ satisfying
\eqref{eq:QQrelations} if and only if the Bethe equations hold.  The
uniqueness of $\widetilde{Q}_i(z)$ follows from the uniqueness of the coefficients of $\tilde{g}_i(z)$.
\end{proof}


\noindent {\bf Remark.} It is worth noting, that the operators, whose eigenvalues are $\widetilde{Q}_i$ are not just auxiliary, but have a geometric meaning. Namely, they correspond to the generating functions of quantum tautological classes of exterior powers of the flop flag variety, i.e.  $V^{\vee}_i$, so that the sequence $0\to V_i\to W_{n-1}\to V_i^{\vee}\to 0$ is exact.\\

From now on let us study solutions of Bethe equations corresponding to complete flag, namely for $\textbf{v}_{i}=i$ in \eqref{eq:XXZGen} and $\textbf{w}_{n-1}=n$.
\begin{proposition}
The system of equations \eqref{eq:TQBaxter} for $\textbf{v}_i=i$ is equivalent to the following system of equations
\begin{equation}
\zeta_{i+1} Q^{(1)}_i \widetilde Q^{(-1)}_i -  \zeta_{i} Q^{(-1)}_i \widetilde Q^{(1)}_i = (\zeta_{i+1}-\zeta_i)Q_{i-1}Q_{i+1}\,.
\label{eq:QQrelationsComplete}
\end{equation}
\end{proposition}

Indeed, in this case $\Delta_i=0$ and $\widetilde{z}_i=\zeta_i$.

\subsection{Construction of the tRS Lax Matrix}
Here we shall find a general solution for $Q_i$ (the so-called Baxter polynomials) which solves \eqref{eq:QQrelationsComplete}. First we need to prove the following (variables $\gamma_1,\dots,\gamma_{k-1}$ are assumed to be formal variables)

\begin{lemma}\label{Th:existencePoly}
 Let $f_1,\dots,f_{k-1}$ be
  polynomials that do not vanish at $0$,
  and let $g$ be an arbitrary polynomial.
   Then there exists unique polynomial  $f_k$ satisfying 
\begin{equation}\label{Fmatrix}
g=\det\begin{pmatrix} f_1 & \gamma_{1} f_1^{(-2)} & \cdots &
  \gamma_{1}^{k-1} f_1^{(-2-2k)} \\ \vdots & \vdots & \ddots & \vdots \\
  f_k  & \gamma_{k} f_k^{(-2)} & \cdots & \gamma_k^{k-1} f_k^{(-2-2k)} 
\end{pmatrix}
\end{equation}
where the numbers in the parentheses in the superscripts denote multiplicative shifts of the argument of the corresponding polynomials, i.e. $f_i^{(-2)}(u)=f_i(\hbar u)$. Moreover, if $g(0)\ne 0$, then $f_k(0)\ne 0$.
 \end{lemma}

\begin{proof} 
Let $V(\gamma_1,\dots,\gamma_k)$ denote the $k\times k$ Vandermonde
matrix.  
\begin{equation}\label{eq:Vandmatrix}
V(\gamma_1,\dots,\gamma_k)=\begin{bmatrix} \, 1 & \gamma_{i_1} & \cdots & \gamma_{i_1}^{j-1} \\ \vdots & \vdots & \ddots & \vdots \\ 1 & \gamma_{i_j} & \cdots & \gamma_{i_j}^{j-1} \,  \end{bmatrix}\,.
\end{equation}
We recall that this determinant is nonzero if
and only if the $\gamma_i$'s are distinct.

Set $f_j(z)=\sum a_{ji} z^i$ and $g(z)=\sum b_i z^i$,
   and let $F$ denote the matrix in \eqref{Fmatrix}.  We show that we
   can find $a_{kj}$'s recursively.  Expanding by minors
   along the bottom row, we get $g=\sum_{j=1}^k (-1)^{k+j}\det
   F_{k,j}f_k^{(j-1)}$.  First we equate the constant terms.  This
   gives
$$
b_0=a_{k0}\left(\prod_{j=1}^{k-1}a_{j0}\right)\sum_{j=1}^k  (-1)^{k+j}\gamma_k^{j-1}\det
   V(\gamma_1,\dots,\gamma_k)_{k,j}=a_{k0}\left(\prod_{j=1}^{k-1}a_{j0}\right)\det
   V(\gamma_1,\dots,\gamma_k)\,.
$$ 
Since the $\gamma_j$'s are distinct,
   the Vandermonde determinant is nonzero.  Moreover, $a_{j0}\ne 0$
   for $j=1,\dots, k-1$.  Thus, we can solve uniquely for $a_{k0}$.
   In particular, if $b_0=0$, then $a_{k0}=0$.

Now we need to make an inductive step. Assume that we have found unique $a_{kr}$
   for $r<s$ such that the polynomial equation \eqref{Fmatrix} has
   equal coefficients up through degree $s-1$.  We now look at the
   coefficient of $z^s$.  The only way that $a_{ks}$ appears in this
   coefficient is through the constant terms of the minors $F_{k,j}$.
   To be more explicit, equating the coefficient in front of $z^s$ in
   \eqref{Fmatrix} expresses $c \,a_{ks}$ as a polynomial in known
   quantities, where
\begin{align*} c&=\left(\prod_{j=1}^{k-1}a_{j0}\right)\sum_{j=1}^k  (-1)^{k+j}(\hbar^{-s}\gamma_k)^{j-1}\det
   V(\gamma_1,\dots,\gamma_{k-1},\hbar^{-s}\gamma_k)_{k,j}\\&=\left(\prod_{j=1}^{k-1}a_{j0}\right)\det
   V(\gamma_1,\dots,\gamma_{k-1},\hbar^{-s}\gamma_k).
 \end{align*}
The condition on the $\gamma_j$'s implies that the Vandermonde
determinant is nonzero, so there is a unique solution for $a_{ks}$. 
 \end{proof}

\begin{proposition}\label{Th:QQVander}
Given polynomials $Q_j, \widetilde Q_j$ for $j=1,\dots,n$ satisfying \eqref{eq:QQrelationsComplete}, there exist unique monic degree one polynomials $q_1,\dots q_n$ such that
\begin{equation}
Q_j(u)= \frac{\text{det}\Big( M_{1,\ldots,j} \Big)}{\text{det}\Big( V_{1,\ldots,j} \Big)}\,,
\qquad
 \widetilde Q_j(u)= \frac{\text{det}\Big( M_{1,\ldots,j-1,j+1} \Big)}{\text{det}\Big( V_{1,\ldots,j-1,j+1} \Big)}\,,
\label{eq:QPolyM}
\end{equation}
where
\begin{equation}
M_{i_1,\ldots,i_j} = \begin{bmatrix} \,  q_{i_1}^{(j-1)} & \zeta_{i_1} q_{i_1}^{(j-3)} & \cdots & \zeta_{i_1}^{j-1} q_{i_1}^{(1-j)} \\ \vdots & \vdots & \ddots & \vdots \\  q_{i_j}^{(j-1)}  & \zeta_{i_j} q_{i_j}^{(j-3)} & \cdots & \zeta_{i_j}^{j-1} q_{i_j}^{(1-j)} \, \end{bmatrix}\,,
\qquad
V_{i_1,\ldots,i_j} =V(\zeta_{i_1},\dots,\zeta_{i_j})\,,
\label{eq:MM0Ind}
\end{equation}
where the Vandermonde matrix in the last equation is given by \eqref{eq:Vandmatrix}.
\end{proposition}

\begin{proof}
Let us observe that since $P$ and $Q_k$ \eqref{eq:QPBaxter} do not vanish at $0$ since Bethe roots and equivariant parameters are formal variables: 
$Q_k(0)\ne 0$ for all $k$.  This implies that $\widetilde Q_k(0)\ne 0$ for all $k$ as well; otherwise, by  \eqref{eq:QQrelationsComplete}, either $Q_{k-1}$ or $Q_{k+1}$ would vanish at $0$.

One can then see that the desired structure of Baxter polynomials $Q_k$ and $\widetilde Q_k$ emerges if we solve the equations iteratively. From the first equation from \eqref{eq:QQrelationsComplete} we get
\begin{equation}
Q_2 = \frac{\zeta_2 Q^{(1)}_1 \widetilde Q^{(-1)}_1 - \zeta_1 Q^{(-1)}_1 \widetilde Q^{(1)}_1}{\zeta_2-\zeta_1}\,.
\label{eq:Q2baseInduction}
\end{equation}
In what follows we relabel $Q_1=q_1$ and $\widetilde Q_1=q_2$. From the formula it is obvious that both polynomials $q_1$ and $q_2$ are monic of degree one:
$$
q_1 = u - p_1,\qquad q_2 = u - p_2\,,
$$
where due to the above reasoning their roots $p_1$ and $p_2$ are nonzero complex numbers.



Next, suppose that for $2\le k\le n-1$, we have shown that there exist
unique polynomials $q_1,\dots,q_{k}$ such the formulae for
$Q_j$ (resp. $\widetilde{Q}_j$) in \eqref{eq:QPolyM} hold for
$1\le j\le k$ (resp. $1\le j\le k-1$).  Furthermore, assume that none
of these polynomials vanish at $0$.  We will show that there exists a
unique $q_{k+1}$ such that the formulae for $Q_{k+1}$ and
$\widetilde{Q}_{k}$ hold and that $q_{k}(0)\ne 0$. This will
prove the lemma.

We use Lemma~\ref{Th:existencePoly} to define $q_{k+1}$.  In the
notation of that lemma, set $f_j=q_{j}^{(k)}$ and
$\gamma_j=\zeta_j$ for $1\le j\le k$, and set
$g=(\det
V_{1,\dots,k})\widetilde{Q}_{k}$.  
By the inductive assumption $f_j(0)\ne 0$ for $1\le j\le k$, so there exists a
unique $f_{k}$ satisfying \eqref{Fmatrix}.  Moreover, $g(0)\ne 0$,
so $f_{k}\ne 0$.  It is now clear that
$q_{k+1}=f_{k+1}^{(k+1)}$ is the unique polynomial
satisfying the formula in \eqref{eq:QPolyM} for
$\widetilde{Q}_{k}$.  Clearly $q_{k+1}(0)\ne 0$.

To complete the inductive step, it remains to show that the formula
for $Q_{k+1}$ is satisfied.  
Recall that
\begin{equation}
\text{det}(\zeta_i^{j-1})=\prod_{1\leq i<j\leq k}(\zeta_i-\zeta_j)\,,\qquad i,j=1,\dots,k\,,
\end{equation}
which will be also the value for $\text{det}(V_{i_1,\dots,i_k})$ and
\begin{equation}
\text{det}(V_{i_1,\dots,i_{k-1},i_{k+1}})=\prod_{1\leq i<j\leq k-1}(\zeta_i-\zeta_j)\prod_{l=1}^{k-1}(\zeta_l-\zeta_{k+1})\,,\qquad i,j=1,\dots,k\,.
\end{equation}
Now we can plug in $Q$-polynomials from \eqref{eq:QPolyM} into \eqref{eq:QQrelationsComplete}, which we want to verify.
Using the above formulae for the Vandermonde determinants we see that \eqref{eq:QQrelationsComplete} is reduced to
\begin{equation}
\zeta_{k+1} \text{det}\,M^{(1)}_{1,\ldots,k}\cdot \text{det}\,M^{(-1)}_{1,\ldots,k-1,k+1}-  \zeta_k \text{det}\,M^{(-1)}_{1,\ldots,k}\cdot\text{det}\, M^{(1)}_{1,\ldots,k-1,k+1}= \text{det} \,M_{1,\ldots,k-1}\cdot\text{det}\,M_{1,\ldots,k+1}\,,
\label{eq:MMrelationsComplete}
\end{equation}
where we recall that the numbers in the parentheses denote multiplicative shifts of the argument of $q$-polynomials by $\hbar^{-\frac{1}{2}}$.
We will now prove that \eqref{eq:MMrelationsComplete} is equivalent to the Desnanot-Jacobi\footnote{Desnanot-Jacobi-Dodgson/Lewis Caroll identity} determinant identity for matrix $M_{1,\ldots,j+1}$, which can be written as follows
\begin{equation}
\text{det}\,M_{k-1}^1\cdot \text{det}\,M_{k}^{k}- \text{det}\,M_{k}^1\cdot \text{det}\,M_{k-1}^{k}= \text{det} \,M_{1,k}^{k-1,k}\cdot\text{det}\,M\,.
\label{eq:JacobiNew}
\end{equation}
Here we denoted $M=M_{1,\ldots,k+1}$, $(k+1)\times(k+1)$ matrix of the form \eqref{eq:MM0Ind}, and $M_{b}^{a}$ is a submatrix which is obtained from $M$ by removing $a$-th row and $b$-th column. Note that \eqref{eq:MMrelationsComplete} has shifts of the $q$-polynomials, whereas \eqref{eq:JacobiNew} does not. However, due to the periodic structure in the columns we can relate shifted $k\times k$ matrices from \eqref{eq:MMrelationsComplete} with submatrices of $M$. Moreover, one can see that
\begin{equation}
M^{(1)}_{1,\ldots,k}=M_k^k\,,\qquad M^{(1)}_{1,\ldots,k-1,k+1} = M_{k-1}^k\,,
\end{equation}
but other matrices do not match directly, albeit they look similar. Let us multiply both sides of \eqref{eq:MMrelationsComplete}  by $\prod_{l=1}^{k-1}\zeta_l$. Then we can absorb this product on the left into matrices $M^{(-1)}_{1,\ldots,k-1,k+1}$ and $M^{(-1)}_{1,\ldots,k}$ by multiplying each of its first $k-1$ rows by $\zeta_i,\,i=1,\dots,k-1$; while on the right we absorb it into matrix $M_{1,\ldots,k-1}$. Additionally in the left hand side we absorb $\zeta_{k+1}$ into the last row of $M^{(-1)}_{1,\ldots,k-1,k+1}$ and $\zeta_k$ into the last row of $M^{(-1)}_{1,\ldots,k}$. To summarize
\begin{equation}
\zeta_{k+1}\prod_{l=1}^{k-1}\zeta_l\,\cdot\text{det}\,M^{(-1)}_{1,\ldots,k-1,k+1}=\text{det}\,M_{k-1}^1\,,\qquad \zeta_k \prod_{l=1}^{k-1}\zeta_l\cdot\text{det}\,M^{(-1)}_{1,\ldots,k}=\text{det}\,M_{k}^1\,,
\end{equation}
and
\begin{equation}
\prod_{l=1}^{k-1}\zeta_l\cdot\text{det} \,M_{1,\ldots,k-1}=\text{det} \,M_{1,k}^{k-1,k}\,,
\end{equation}
so \eqref{eq:MMrelationsComplete} is equivalent to \eqref{eq:JacobiNew}. Therefore $Q\widetilde{Q}$ relations \eqref{eq:QQrelationsComplete} are equivalent to the Desnanot-Jacobi identity provided that \eqref{eq:QPolyM,eq:MM0Ind} hold.
\end{proof}

Finally we are ready to prove the main theorem which relates XXZ Bethe equations with trigonometric RS model.

\begin{theorem}\label{Th:XXZ/tRS}
Let $L$ be the following matrix
\begin{equation}
L_{ij} = \frac{\prod\limits_{k \neq j}^n \left(\hbar^{-^1/_2} \zeta_i \,  - \hbar^{^1/_2} \zeta_k \,  \right) }{\prod\limits^n_{k \neq j} \left( \zeta_j-\zeta_k\right) }   p_i\,.
\label{eq:tRSLAX}
\end{equation}
Then for each eigenvector of the operator of quantum multiplication by 
\begin{equation}\label{eq:MomentaQuatnumMult}
 - \frac{\mathbf{Q}_i(0) }{\mathbf{Q}_{i-1}(0) }=\hbar^{i-\frac{1}{2}}\widehat{\Lambda^iV_i}(z)\circledast\widehat{\Lambda^{i-1}{V^*}_{i-1}}(z)\,,\quad i=1,\dots,n
\end{equation}
the corresponding eigenvalue defines a unique solution of   
\begin{equation}
P(u) = \text{det}\Big(u - L \Big)\,,
\label{eq:tRSW}
\end{equation}
where $P(u)$ is given by \eqref{eq:QPBaxter}. This correspondence establishes a bijection between solutions of \eqref{eq:tRSW} and the above eigenvectors.
\end{theorem}


\begin{proof}
Using Proposition \ref{Th:QQVander} we can put $j=n$ in \eqref{eq:QPolyM}
\begin{equation}
P(u)= \frac{\text{det}\Big( M_{1,\ldots,n} \Big)}{\text{det}\Big( V_{1,\ldots,n} \Big)}\,.
\end{equation}
Let us multiply $i$th column of $M_{1,\ldots,n}$ by $\hbar^{\frac{2i-n-1}{2}}$. Since $\prod_{i=1}^n \hbar^{\frac{2i-n-1}{2}}=1$ the determinant of this matrix will remain unchanged, however, each matrix element will now contain a monic polynomial in $u$ of degree one, while the multiplicative shifts will be applied to its coefficients $p_i$. Let us call this matrix $M'_{1,\ldots,n}(u)$. Notice that
\begin{equation}
M'_{1,\ldots,n}(u) = V_{1,\ldots,n} \cdot u + M'_{1,\ldots,n}(0)\,.
\end{equation}
We can now simplify the formulae by inverting Vandermonde matrix $V_{1,\ldots,n}$ as follows
\begin{equation}
P(u)=\text{det}\left(u\cdot 1-L\right)\,, 
\label{eq:BaxterFlavorPol}
\end{equation}
where 
\begin{equation}
L=-M'_{1,\ldots,n}(0)\cdot\Big( V_{1,\ldots,n} \Big)^{-1}\,.
\label{eq:tRSLaxMatrix}
\end{equation}
Straightforward computation shows that $L$ is provided by \eqref{eq:tRSLAX}. Indeed, the inverse of the Vandermonde matrix reads
\begin{equation}
(V_{1,\dots, n}^{-1})_{t,j} = (-1)^{t+j}\frac{S_{n-t,j}(\zeta_1,\dots,\zeta_n)}{\prod\limits_{l\neq j}^n(\zeta_j-\zeta_l)}\,,
\end{equation}
where 
$$
S_{k,j}(\zeta_1,\dots,\zeta_n)= S_k(\zeta_1,\dots,\zeta_{j-1},\zeta_{j+1},\zeta_n)\,,
$$
and 
$$
S_k(\zeta_1,\dots,\zeta_n)=\sum\limits_{1\leq i_1\leq \dots \leq i_k \leq n}^n \zeta_{i_1}\cdots\zeta_{i_k}\,.
$$
Then we have 
$$
\left(-M'_{1,\ldots,n}(0)\right)_{i, t} = \hbar^{\frac{n+1-2t}{2}}p_i\,.
$$
Thus, according to \eqref{eq:tRSLaxMatrix}
$$
L_{i,j} = \sum_{t=1}^n  \frac{(-1)^{t+j}\hbar^{\frac{n+1-2t}{2}}S_{n-t,j}(\zeta_1,\dots,\zeta_n)}{\prod\limits_{l\neq j}^n(\zeta_j-\zeta_l)} p_i =  \frac{\prod\limits_{m \neq j}^n \left(\hbar^{-^1/_2} \zeta_i \,  - \hbar^{^1/_2} \zeta_m \,  \right) }{\prod\limits_{l\neq j}^n(\zeta_j-\zeta_l)}   p_i\,.
$$
\end{proof}

Along the way we have discovered a new presentation of the tRS Lax matrix in terms of products of Vandermonde-type matrices \eqref{eq:tRSLaxMatrix}.

\vskip.1in
It remains to prove that momenta $p_i$ which appear as roots of first degree polynomials $s_i$ are given by formula \eqref{eq:piQformula} which provides geometric meaning of the tRS momenta.

\begin{lemma}
Given $q_i(z)= z - p_i,\, i = 1,\dots, n$ in matrix $M_{1,\dots,n}$ from \eqref{eq:MM0Ind} the following formula 
\begin{equation}\label{eq:piQformula}
p_i= - \frac{Q_i(0) }{Q_{i-1}(0) } = -\frac{\sigma_{i,1}\cdots\sigma_{i,i}}{\sigma_{i-1,1}\cdots\sigma_{i-1,i-1}}\,.
\end{equation}
\end{lemma}

\begin{proof}
Let us evaluate matrix $M_{1,\dots, k}(z)$ at $z=0$ for $k=1,\dots,n$. The following immediately follows
\begin{equation}
M_{1,\dots, k}(0)=-\text{diag}(p_{1},\dots,p_{k})\cdot  V_{1,\dots, k}\,,
\end{equation}
since $\left(M_{1,\dots, k}(0)\right)_{i,j}=-\zeta_i^{j-1} p_{i}$ and $\left(V_{1,\dots, k}\right)_{i,j}=\zeta_i^{j-1}$. Therefore, according to Proposition \ref{Th:QQVander} 
\begin{equation}\label{eq:detMV}
Q_k(0)=\text{det}\left[M_{1,\dots, k}(0)\cdot V_{1,\dots, k}^{-1}\right]=(-1)^k p_{1}\cdots p_{k}\,.
\end{equation}
which proves formula \eqref{eq:piQformula}.

\end{proof}

Matrix $L$ is known as Lax matrix for the trigonometric Ruijsenaars-Schneider model\footnote{In the literature slightly different normalizations are sometimes used.}. The theorem shows that its characteristic polynomial is equal to Baxter polynomial $P(u)$ whose roots are equivariant parameters $a_1,\dots, a_n$. By expanding both sides of \eqref{eq:tRSW} in $u$ we find explicitly the tRS Hamiltonians $H_1,\dots, H_n$
\begin{equation}
 \text{det}\left(u\cdot 1 -  L(\zeta_i, p_i,\hbar) \right) = \sum_{r=0}^n (-1)^rH_r(\zeta_i, p_i,\hbar) u^{n-r}\,,
\label{eq:tRSLaxDecomp}
\end{equation}
are equal to the corresponding elementary symmetric functions of the equivariant parameters
\begin{equation}
\label{eq:tRSEnergyEqns}
H_r(\zeta_i, p_i,\hbar)= e_r(\mathsf{a}_1,\dots, \mathsf{a}_n)\,,
\end{equation}
where
\begin{equation}
e_r(\mathsf{a}_1,\dots, \mathsf{a}_n)= \sum_{\substack{\mathcal{I}\subset\{1,\dots, n\} \\ |\mathcal{I}|=r}}\prod_{k\in\mathcal{I}}\mathsf{a}_k\,.
\label{eq:tRSEigenvalues}
\end{equation}

The phase space of the tRS model is described as follows.
Parameters $\zeta_1,\dots,\zeta_n$ and their conjugate momenta $p_1,\dots, p_n$ serve as canonical coordinates on the cotangent bundle to $\left(\Complex^\times\right)^n$. The symplectic form reads
\begin{equation}
\Omega = \sum_{i=1}^n \frac{dp_i}{p_i}\wedge \frac{d\zeta_i}{\zeta_i}\,.
\end{equation}

\noindent \textbf{Remark.}
It was shown in \cite{Bullimore:2015fr} classical momenta $p_i$ can be determined from the (exponentials of) derivatives of the so-called Yang-Yang function\footnote{Bethe equations arise as derivatives of the Yang-Yang function with respect to all Bethe roots $\sigma_{i,k}$.} for Bethe equations \eqref{eq:BetheEqDetailed}. These defining relations describe a complex Lagrangian submanifold
$\mathcal{L}\subset T^*\left(\Complex^\times\right)^n$, such that the generating function for this submanifold ($\Omega$ is identically zero on $\mathcal{L}$) is given by the Yang-Yang function. It is important to mention that relation of the spectrum of XXZ spin chains to Yang-Yang function was previously 
noted in the study of quantum Knizhnik-Zamolodchikov equations \cite{Tarasov:}.

\begin{proposition}
The Hamiltonians of the $n$-body tRS model are given by
\begin{equation}\label{eq:tRSRelationsEl}
H_r=\sum_{\substack{\mathcal{I}\subset\{1,\dots,n\} \\ |\mathcal{I}|=r}}\prod_{\substack{i\in\mathcal{I} \\ j\notin\mathcal{I}}}\frac{\zeta_i\, \hbar^{-^1/_2}-\zeta_j \, \hbar^{^1/_2}}{\zeta_i-\zeta_j}\prod\limits_{k\in\mathcal{I}}p_k \,,
\end{equation}
where $r=0,1,\ldots,n$.
In particular,
\begin{equation}
H_1=\text{Tr}\,L=\sum_{i=1}^n\prod_{j\neq i}^n\frac{\zeta_i\, \hbar^{-^1/_2}-\zeta_j \, \hbar^{^1/_2}}{\zeta_i-\zeta_j} p_i \,,\qquad H_n =\text{det}\, L= \prod\limits_{k=1}^n p_k \,.
\label{eq:tRSHam1}
\end{equation}
\end{proposition}

Note that $H_1,\dots, H_r$ coincide with the classical version of the Macdonald difference operators.

\begin{proof}
Let us first see how the proposition works in the case of $2\times 2$ matrix, i.e. $n=2$. In this case the $L$-matrix looks like this:
\[ \left( \begin{array}{cc}
\frac{\hbar^{-^1/_2}\zeta_1-\hbar^{^1/_2}\zeta_2}{\zeta_1-\zeta_2}p_1 & \frac{\hbar^{-^1/_2}\zeta_1-\hbar^{^1/_2}\zeta_1}{\zeta_1-\zeta_2}p_2 \\
&\\
 \frac{\hbar^{-^1/_2}\zeta_2-\hbar^{^1/_2}\zeta_2}{\zeta_1-\zeta_2}p_1&
  \frac{\hbar^{-^1/_2}\zeta_2-\hbar^{^1/_2}\zeta_1}{\zeta_2-\zeta_1}p_2
\end{array} \right)\]
An elementary calculation shows that the statement is true and in particular, the determinant of this matrix is equal to $p_1p_2$ due to the fact that the second order pole in $(\zeta_1-\zeta_2)$ disappear. This will be relevant in the case of higher $n$.

To prove the statement in the case of general $n$ we use the Fredholm decomposition:
\begin{eqnarray}
\text{det}\left(u\cdot 1- L \right)=\sum^n_{r=0}u^{n -r}(-1)^r \text{Tr}\,\Lambda^r(L)\,,
\end{eqnarray}
where $\Lambda^r$ denotes the exterior power. Clearly, $\text{Tr}\,\Lambda^r(L)$ is just the sum over all minors of rank
r. Let us look at the terms representing each minor in detail.
The explicit expression for each of them is given by the sum over the products of its matrix elements accompanied by a sign. It is easy to see that the common divisor for such products is exactly
\begin{eqnarray}\label{comdiv}
\prod_{\substack{i\in\mathcal{I} \\ j\notin\mathcal{I}}}\frac{\zeta_i\, \hbar^{-^1/_2}-\zeta_j \, \hbar^{^1/_2}}{\zeta_i-\zeta_j}\prod\limits_{k\in\mathcal{I}}p_k \,,
\end{eqnarray}
where $\mathcal{I}$ is the number of indices representing the minor. Other terms involve products with poles $(\zeta_i-\zeta_j)$ where both $i, j$ belong to $\mathcal{I}$. Let us show that all of these poles disappear as in the $2\times 2$ case.
Note, that such pole $(\zeta_i-\zeta_j)$ appears twice in each product. Let us show that there is no such pole in the final expression. To do that let us expand each minor using the row decomposition till we reach the $2\times 2$ minor $L_{\{i,j\}}$. Clearly, this is the only term in this expansion containing such a pole, and by the same calculation as in $2\times 2$  case as above, it cancels out.
Therefore, the coefficient of (\ref{comdiv}) in the expansion does not depend on $\zeta_i$ as one can deduce from counting the powers of
$\zeta_i$ in the numerator and the denominator.
To finish the proof one needs to show that the resulting constant is equal to 1 for any $\mathcal{I}$. That is clear from the normalization of ``non-difference terms'', in numerator, which are responsible for pole cancellation, namely  $\zeta_i(\hbar^{-^1/_2}-\hbar^{^1/_2})$.
\end{proof}

We are now ready to formulate the main theorem of this section. In \eqref{eq:zvdzeta} we can put $\zeta_n=1$ and express the variables $\zeta_j$ via $z_i$ as
\begin{equation}
\label{eq:KahlerParameterszeta}
\zeta_j = \prod_{l\geq j} z_l\,,\qquad \frac{\zeta_i}{\zeta_j}=z_{i+1}\cdots z_{j-1}\,,\qquad j>i\,.
\end{equation}
One can notice that the Hamiltonians $\{H_r\}$ depend only on the momenta $\{p_i\}$ and the ratios of the  coordinates $\{\zeta_i\}$ and thus are the functions of the variables $\{p_i\}, \{z_i\}$. 
\begin{theorem}\label{Th:KTmain}
The quantum equivariant K-theory of the cotangent bundle to complete n-flag is given by
\begin{equation}
QK_T(T^*\mathbb{F}l_n)=\frac{\Complex[\mathsf{a}^{\pm 1}_1,\dots,\mathsf{a}^{\pm 1}_n,\hbar^{\pm 1/2},\, p_1^{\pm 1},\dots,p_n^{\pm 1}][[z_1,\dots,z_{n-1}]]}{(R_1,\dots,R_{n})}\,,
\end{equation}
where $R_1,\dots,R_{n}$ are the coefficients of the following polynomial:\footnote{Notice that relations $R_i=0$ are nothing but tRS energy level equations $H_i=e_i$ (see \eqref{eq:tRSEnergyEqns}) multiplied by common denominator for each $i=1,\dots, r$.}
\begin{equation}
\text{det}(M_{1,\dots, n})(u)-P(u)\cdot \text{det}(V_{1,\dots, n}) = \sum_{k=0}^n (-1)^k u^{n-k} R_k\,.
\end{equation}
\end{theorem}

\begin{proof}
The statement directly follows from Proposition \ref{generators}, the fact that coefficients of ${\bf Q}_i$-operators are generators of all tautological bundles, and Theorem \ref{Th:XXZ/tRS}. 
\end{proof}

\noindent {\bf Remark.} We mention that in \cite{Rimanyi:2014ef} (section 13) the authors conjectured the generators and 
relations for quantum K-theory of cotangent bundles of flag varieties. In the case at hand (full flags) we indicate that our formulas do coincide, thus proving the conjecture of \cite{Rimanyi:2014ef}.

We also indicate that the relations between various limits of spin chain models and many body systems were studied extensively in recent years within integrable systems community, see e.g. \cite{Mukhin:}, \cite{Zabrodin:}.

\subsection{Dual tRS Model from XXZ Chain}
In \eqref{eq:tRSRelationsEl} tRS Hamiltonians are functions of quantum parameters $\zeta_1,\dots, \zeta_n$ and the eigenvalues \eqref{eq:tRSEigenvalues} are given by symmetric polynomials of equivariant parameters. It turns out that there is a dual formulation of the integrable model such that these parameters switch roles and is know as \textit{bispectral duality}. We can show that from starting from Bethe equations \eqref{eq:XXZGen} we can derive the dual set of tRS Hamiltonians.

\begin{theorem}\label{Th:XXZ/tRSdual}
Let $L^{!}$ be the following matrix
\begin{equation}
L^{!}_{ij} = \frac{\prod\limits_{k \neq j}^n \left(\hbar^{^1/_2} \mathsf{a}_i \,  - \hbar^{-^1/_2} \mathsf{a}_k \,  \right) }{\prod\limits^n_{k \neq i} \left( \mathsf{a}_i-\mathsf{a}_k\right) }   p^{!}_j\,,
\label{eq:tRSLAXdual}
\end{equation}
where
\begin{equation}
p^{!}_j = \hbar^{n-1}\zeta_n \frac{Q^{(1)}_{n-1}(\mathsf{a}_j) }{Q^{(-1)}_{n-1}(\mathsf{a}_j) }\,,\quad j=1,\dots,n\,.
\end{equation}
Then Bethe equations \eqref{eq:XXZGen} are equivalent to
\begin{equation}
H_r^!:=\text{Tr}\,\Lambda^r(L^!) = e_r(\zeta_1,\dots,\zeta_n)\,.
\label{eq:tRSWdual}
\end{equation}
\end{theorem}

In other words, diagonalization of the Lax matrix $L^{!}$ of the dual $n$-body tRS model is equivalent to solving the same Bethe equations \eqref{eq:XXZGen}.
We can see that bispectral duality works in the following way on the level of Lax matrices
\begin{equation}
L^{!}_{ij} (\mathsf{a}_1,\dots,\mathsf{a}_n;\zeta_1,\dots,\zeta_n,\hbar; p^{!}_1,\dots,p^{!}_n)=L_{ij} (\zeta_1,\dots,\zeta_n; \mathsf{a}_1,\dots,\mathsf{a}_n,\hbar^{-1}; p_1,\dots,p_n)\,.
\end{equation}
Note that in order to prove this theorem it is sufficient to study only one of the Hamiltonians, say $H_1^!$. Eigenvalues of the other Hamiltonians follow from integrability of the model.

\begin{proof}
See Appendix A2 of \cite{Gaiotto:2013bwa}.
\end{proof}

\noindent \textbf{Remark.}
Theorem \ref{Th:KTmain} shows that defining relations for quantum K-theory of $T^*\mathbb{F}l_n$ are equivalent to integrals of motion of the $n$-body tRS model. One can therefore ask what which integrable system describes K-theory of cotangent bundles to partial flags $T^*G/P$. A procedure which leads to the answer was outlined in \cite{Gaiotto:2013bwa,Bullimore:2015fr} and states the following. One can define a restricted tRS model such that positions of some of its particles are fixed relative to each other. This restriction defines parabolic subgroup $P\subset G$. Thus the unrestricted tRS phase space corresponds arises when parabolic subgroup $P$ is replaced with a Borel subgroup $B$.
Using a chain of specifications of equivariant parameters or quantum parameters of $QK_T(T^*\mathbb{F}l_n)$ we can arrive to quantum K-ring of the desired Nakajima variety $T^*G/P$. Depending on which set of parameters is chosen, the original tRS Hamiltonians/Macdonald operators \eqref{eq:tRSRelationsEl}, or their dual counterparts \eqref{eq:tRSWdual}, will be used to define the corresponding K-rings. In recent mathematical literature some progress in this direction was made in \cite{Rimanyi:2014ef}. We plan to return to these questions in the near future.\\

\noindent \textbf{Remark.}
The XXZ/tRS duality, which we have developed in this section, is a top element in the hierarchy of \textit{spin chain/many-body system} dualities, which were outlined in \cite{Gaiotto:2013bwa}. A XXZ spin chain can be reduced to either the XXX spin chain or to the Gaudin model. On the other side, a trigonometric Ruijsenaars-Schneider model can be reduced to the rational Ruijsenaars-Schneider model or to the trigonometric Calogero-Moser model. The dualities on this level of hierarchy will provide a description for quantum equivariant cohomology of $T^*G/B$ in terms of integrals of motion of rational RS or trigonometric Calogero-Moser model, analogously to our quantum K-theory statement (Theorem \ref{Th:XXZ/tRS}).

\section{Compact Limit of XXZ Bethe Ansatz and of tRS Model}\label{Sec:tRSToda}

In this final section we shall compare our results, 
with the work of Givental and Lee \cite{2001math8105G} where they give a description of the equivariant quantum K-theory ring of complete flag varieties. 
In contrast with the approach of the present paper,
the ring constructed in \cite{2001math8105G} utilises the moduli spaces of stable maps. We denote the Givental-Lee quantum K-theory ring by $QK^{GL}_{T'}(\mathbb{F}l_n)$ (here $T'$ is the maximal torus of $U(n)$). We show that in the limit $\hbar \to \infty$ the quantum K-theory ring  $QK_T(T^{*}\mathbb{F}l_n)$  which we study in this paper degenerates to a ring isomorphic to $QK^{GL}_{T'}(\mathbb{F}l_n)$. 

In order to understand quantum multiplication in $QK_{T'}(T^{*}\mathbb{F}l_n)$ at $\hbar\to \infty$ we must compute the corresponding limit of Bethe equations \eqref{eq:XXZBetheGeom} which is given in \eqref{eq:XXZCompact}. Then we will need to  follow the steps of \secref{Sec:XXZ/tRS} and present the resulting Bethe equations as conditions for roots of a characteristic polynomial of a certain matrix, which will appear to be the Lax matrix of difference Toda model \cite{Etingof:ac}. In this way we arrive at relations in $QK^{GL}_{T'}(\mathbb{F}l_n)$ from \cite{2001math8105G}.

\subsection{Five-Vertex Model and Quantum Toda Chain}
Using Baxter Q-polynomials we can present Bethe equations \eqref{eq:XXZCompact} in a more concise form.
\begin{lemma}
Let
\begin{equation}
Q_i(u)=\prod_{j=1}^i(u-s_{i,j})\,, \qquad M(u):=Q_n(u)=\prod_{i=1}^n(u-a_i)\,.
\end{equation}
Then we can rewrite \eqref{eq:XXZCompact} as
\begin{equation}\label{eq:XXZCompactBaxter}
\frac{Q_{i+1}(s_{i,k})}{Q_{i-1}(s_{i-1,k})}\cdot\frac{\prod\limits_{j=1}^{\textbf{v}_i} s_{i,j}}{\prod\limits_{j=1}^{\textbf{v}_{i+1}} s_{i+1,j}}=z^{\#}_i (-1)^{\delta_i}(s_{i,k})^{\textbf{v}_i-\textbf{v}_{i-1}-1}\,,
\end{equation}
\end{lemma}
where $\delta_i$ are given after the formula \eqref{eq:BetheEqDetailed}. As in the previous section we shall focus on complete flag varieties for which $\textbf{v}_i=i$, thus the exponent of $s_{i,k}$ in the right hand side of the above expression vanishes.\\

\noindent \textbf{Remark.} Equations \eqref{eq:XXZCompact} generalize the result of \cite{Kim:1995us} and serve as Bethe ansatz equations for the five-vertex model. \\

\noindent Using auxiliary Baxter polynomials we can rewrite \eqref{eq:XXZCompactBaxter} in the $Q\widetilde Q$ form similarly to \eqref{eq:QQrelationsComplete}.
\begin{proposition}\label{eq:PropositionqToda}
The system of equations \eqref{eq:XXZCompactBaxter} for $\textbf{v}_i=i$ is equivalent to the following system
\begin{equation}\label{eq:QQtToda}
Q_{i+1}(u)-\frac{\mathfrak{z}_{i+1}}{\mathfrak{z}_i}Q_{i-1}(u)\cdot u \cdot \mathfrak{p}_{i+1} = Q_i(u)\widetilde{Q}_i(u)\,,\quad i=1,\dots,n
\end{equation}
where $z^{\#}_i=\frac{\mathfrak{z}_{i}}{\mathfrak{z}_{i+1}}$, $\widetilde{Q}_i(u),\, i=1,\dots, {n-1}$ are monic polynomials of degree one and
\begin{equation}
\mathfrak{p}_i = - \frac{Q_i(0) }{Q_{i-1}(0) }\,.
\label{eq:TodaLaxMomenta}
\end{equation}
\end{proposition}

\begin{proof}
Analogous to the proof of Proposition \ref{eq:PropQQtilde}.
\end{proof}

We can now formulate a statement which connects the five-vertex models with the q-Toda chain in the same way as the XXZ spin chain is dual to the tRS model (Theorem \ref{Th:XXZ/tRS}).
\begin{theorem}\label{Th:QtodaLax}
System of equations \eqref{eq:QQtToda} is equivalent to
\begin{equation}
M(u)=\det A(u)\,,
\end{equation}
where $A(u)$ is the Lax matrix of the difference Toda chain. It has the following nonzero elements
\begin{equation}
A_{i+1,i}=1\,,\quad A_{i,i}=u-\mathfrak{p}_i\,,\quad A_{i,i+1}=-u\frac{\mathfrak{z}_{i+1}}{\mathfrak{z}_{i}}\mathfrak{p}_{i+1}\,.
\end{equation}
\end{theorem}

\begin{proof}
This statement can be readily proven along the lines of Theorem \ref{Th:XXZ/tRS}.
\end{proof}

\subsection{Compact Limit of tRS Model}
Note that the q-Toda Lax matrix $A(u)$ cannot be obtained as a scaling limit of the tRS Lax matrix \eqref{eq:tRSLAX}. However, one can directly compute q-Toda Hamiltonians from tRS Hamiltonians \eqref{eq:tRSHam1}. This limit was already discussed in the literature (see e.g. \cite{Gerasimov:2008ao} p.13). In our notations this limit can be implemented as follows. First we rescale tRS coordinates, momenta \eqref{eq:tRSRelationsEl} and equivariant parameters \eqref{eq:tRSEigenvalues} as follows
\begin{equation}
\mathfrak{z}_i= \hbar^{-i}\zeta_i\,,\qquad \mathfrak{p}_i = \hbar^{-i+1/2}p_i\,,\qquad \mathfrak{a}_i=\hbar^{-\frac{n}{2}}\mathsf{a}_i=a_i\,.
\end{equation}
Recall that tRS Hamiltonians \eqref{eq:tRSRelationsEl} were derived from XXZ Bethe equations \eqref{eq:BetheEqDetailed} after rescaling of the parameters \eqref{eq:rescaling}. Therefore, in order to restore the original notations of earlier sections, we need to take this into account. In particular, the new equivariant parameters $\mathfrak{a}_i$ coincide with the original $a_i$ parameters, whereas the new momenta $\mathfrak{p}_i$ reproduce \eqref{eq:TodaLaxMomenta}.

Second, after taking $\hbar\to\infty$ limit, we obtain q-Toda Hamiltonian functions which are equal to symmetric polynomials of $\mathfrak{a}_i$
\begin{equation}
H^{\text{q-Toda}}_r(\mathfrak{z}_1,\dots\mathfrak{z}_n;\mathfrak{p}_1,\dots, \mathfrak{p}_n)=e_r(\mathfrak{a}_1,\dots, \mathfrak{a}_n)\,,
\end{equation}
where the Hamiltonians are
\begin{equation}\label{eq:tRSRelationsElToda}
H^{\text{q-Toda}}_r=\sum_{\substack{\mathcal{I}=\{i_1<\dots<i_r\} \\ \mathcal{I}\subset\{1,\dots,n\}  }}\prod_{\ell=1}^r\left(1-\frac{\mathfrak{z}_{i_{\ell}-1}}{\mathfrak{z}_{i_\ell}}\right)^{1-\delta_{i_{\ell}-i_{\ell-1},1}}\prod\limits_{k\in\mathcal{I}}\mathfrak{p}_k \,,
\end{equation}
where $i_{0}=0$. For instance, the first Hamiltonian reads
\begin{equation}
H_1^{\text{q-Toda}}=\mathfrak{p}_1 +\sum\limits_{i=2}^n \mathfrak{p}_i \left(1-\frac{\mathfrak{z}_{i-1}}{\mathfrak{z}_{i}}\right)\,.
\end{equation}

Thus we have shown that the $\mathfrak{gl}(n)$ five-vertex model is dual to the difference Toda $n$-body system such that Bethe equations of the former \eqref{eq:XXZCompact} can be rewritten as equations of motion of the latter.

Finally we can formulate the main statement of this section. Analogously with Theorem \ref{Th:KTmain} we want to state the following theorem in terms of \textit{bona fide} K\"ahler parameters of the flag variety $z^{\#}_i$. We can put $z^{\#}_n=1$ and 
\begin{equation}
\mathfrak{z}_j = \prod_{l\geq j} z^{\#}_l\,,\qquad \frac{\mathfrak{z}_i}{\mathfrak{z}_j}=z^{\#}_{i+1}\cdots z^{\#}_{j-1}\,,\qquad j>i\,,
\end{equation}
and rewrite q-Toda Hamiltonians \eqref{eq:tRSRelationsElToda} via $z^{\#}_i$.

\begin{theorem}\label{todath}
At $\hbar=\infty$ the ring $QK_{T}(T^{*} \mathbb{F}l_n)$ has the following explicit description:
\begin{equation} \label{relatq}
\left.QK_{T}(T^{*}\mathbb{F}l_n)\right|_{\hbar=\infty}=\frac{\Complex\left[z^{\#}_1,\dots,z^{\#}_{n-1},\, \mathfrak{a}_1^{\pm 1},\dots,\mathfrak{a}_n^{\pm 1},\, \mathfrak{p}_1^{\pm 1},\dots,\mathfrak{p}_n^{\pm 1}\right]}{\left(H^{\text{q-Toda}}_r (\{\mathfrak{p}_i\},\{z^{\#}_i\})= e_r(\mathfrak{a}_1,\dots, \mathfrak{a}_n)\right)}\,,
\end{equation}
where $H^{\text{q-Toda}}_r$ are given in \eqref{eq:tRSRelationsElToda}. In particular, this ring is isomorphic to the Givental-Lee quantum K-theory of complete flag varieties from \cite{2001math8105G}: 
$$
\left.QK_{T}(T^{*}\mathbb{F}l_n)\right|_{\hbar=\infty}\cong  QK^{GL}_{T'}(\mathbb{F}l_n).
$$
\end{theorem}

\begin{proof}
After comparison with \eqref{eq:MomentaQuatnumMult} we can see that q-Toda momenta $\mathfrak{p}_i$ geometrically correspond to quantum multiplication by class $\widehat{\Lambda^j{V}_j}(z)\circledast\widehat{\Lambda^{j-1}{V^*}_{j-1}}(z)$ of the flag variety.
Then, (\ref{relatq}) follows from Proposition \ref{eq:PropositionqToda} and Theorem \ref{Th:QtodaLax}. Finally, we observe that the generators and relations in (\ref{relatq}) coincide precisely with those of $ QK^{GL}_{T'}(\mathbb{F}l_n)$ described  in \cite{2001math8105G}.
\end{proof}

\noindent {\bf Remark.} The relations of quantum K-theory ring of flag varieties to relativistic Toda chain was previously discussed by A. Kirillov and T. Maeno in an unpublished work, see also \cite{Lenart:}, \cite{Maeno:}. 
\vskip.1in

While this manuscript has been under review the following papers appeared which further develop the correspondence between quantum K-theory and integrable systems: \cite{Koroteev:2018a,Koroteev:2020,Frenkel:2020,Koroteev:2021}.

\bibliography{cpn2}

\end{document}